\documentclass[12pt,reqno,twoside]{amsart}
\usepackage{amsfonts,amsmath,amssymb}
\usepackage{mathrsfs,mathtools}
\usepackage{enumerate}
\usepackage{hyperref}
\usepackage{esint}
\usepackage{graphicx}
\usepackage{bm}
\usepackage{commath}
\usepackage{esint}
\usepackage{placeins}
\usepackage{flafter}
\usepackage{tikz}
\usepackage[textsize=small]{todonotes}
\usepackage[dvips,bottom=1.4in,right=1in,top=1in, left=1in]{geometry}
\usepackage[latin1]{inputenc}

\DeclareMathAlphabet{\mathpzc}{OT1}{pzc}{m}{it}
\numberwithin{equation}{section}
\makeatletter
\def\eqnarray{\stepcounter{equation}\let\@currentlabel=\theequation
\global\@eqnswtrue
\tabskip\@centering\let\\=\@eqncr
$$\halign to \displaywidth\bgroup\hfil\global\@eqcnt\z@
  $\displaystyle\tabskip\z@{##}$&\global\@eqcnt\@ne
  \hfil$\displaystyle{{}##{}}$\hfil
  &\global\@eqcnt\tw@ $\displaystyle{##}$\hfil
  \tabskip\@centering&\llap{##}\tabskip\z@\cr}
\def\endeqnarray{\@@eqncr\egroup
      \global\advance\c@equation\m@ne$$\global\@ignoretrue}

\newtheorem{theorem}{Theorem}[section]

\newtheorem{corollary}[theorem]{Corollary}
\newtheorem{definition}[theorem]{Definition}
\newtheorem{example}[theorem]{Example}
\newtheorem{lemma}[theorem]{Lemma}

\newtheorem{proposition}[theorem]{Proposition}

\newtheorem{remark}[theorem]{Remark}
\numberwithin{equation}{section}

\thanks{The first author is partially supported by NSF grants DMS-2110263, DMS-1913004 
and the Air Force Office of Scientific Research under Award NO: FA9550-19-1-0036. 
The last author is partially supported by the Air Force Office of Scientific Research under 
Award NO:  FA9550-18-1-0242 and the Army Research Office (ARO) under Award NO:  W911NF-20-1-0115.
}
\input{tcilatex}

\begin{document}
\title[optimal control of fractional in time semilinear PDEs]{{\textsf{A
unified framework for optimal control of fractional in time subdiffusive
semilinear PDEs}}}
\author{Harbir Antil}
\address{H. Antil, The Center for Mathematics and Artificial Intelligence
(CMAI) and Department of Mathematical Sciences, George Mason University,
Fairfax, VA 22030, USA.}
\email{hantil@gmu.edu}
\author{Ciprian G. Gal}
\address{C. G. Gal,Department of Mathematics \& Statistics, Florida
International University, Miami, FL 33199, USA.}
\email{cgal@fiu.edu}
\author{Mahamadi Warma}
\address{M. Warma, The Center for Mathematics and Artificial Intelligence
(CMAI) and Department of Mathematical Sciences, George Mason University,
Fairfax, VA 22030, USA. }
\email{mwarma@gmu.edu}

\begin{abstract}
We consider optimal control of fractional in time (subdiffusive, i.e., for $%
0<\gamma <1$) semilinear parabolic PDEs associated with various notions of
diffusion operators in an unifying fashion. Under general assumptions on the
nonlinearity we{~\textsf{first show}} the existence and regularity of
solutions to the forward and the associated {\textsf{backward (adjoint)}}
problems. In the second part, we prove existence of optimal {\textsf{controls%
}} and characterize the associated {\textsf{first order}} optimality
conditions. Several examples involving fractional in time (and some
fractional in space diffusion) equations are described in detail. The most
challenging obstacle we overcome is the failure of the\ semigroup property
for the semilinear problem in any scaling of (frequency-domain) Hilbert
spaces.
\end{abstract}

\maketitle
\tableofcontents

\section{Introduction}

Optimization problems constrained by partial differential equations (PDEs)
are ubiquitous in science and engineering. Without any specific mention, we
will refer to these problems as optimal control problems. See the monographs 
\cite{FTroeltzsch_2010a, MHinze_RPinnau_MUlbrich_SUlbrich_2009a,
HAntil_DPKouri_MDLacasse_DRidzal_2018a,KIto_KKunisch_2008a} and references
therein for many applications and general results for such problems. In
particular, such optimization problems arise in fluid dynamics,
superconductivity, phase field modeling, regularized variational
inequalities and contact problems, etc. These all are the examples of
optimization problems with parabolic semilinear PDEs as constraints.
Semilinear optimal control problems are known to be a key testbed for
developing new algorithms and/or analysis and there is a significant amount
of literature available on this topic \cite%
{ECasas_1993a,ECasas_1997a,AW-ESAIM, HAntil_DPKouri_DRidzal_2021a}. The
optimization variable (control variable) either acts in the interior
(distributed control), or on the boundary (boundary control), or in the
exterior (exterior control). The first two notions of controls are
well-known, the exterior control is new, see \cite%
{warma2018approximate,HAntil_RKhatri_MWarma_2018a,AVW}.

The goal of this paper is to develop a unified framework for optimal control
problems constrained by fractional in time semilinear PDEs 
\begin{equation}
\left\{ 
\begin{array}{ll}
\partial _{t}^{\gamma }u\left( x,t\right) +Au\left( x,t\right) =f\left(
u\left( x,t\right) \right) +\mathbb{B}z\left( x,t\right) , & \text{in }%
\Omega \times (0,\infty ), \\ 
u\left( \cdot ,0\right) =u_{0}\text{ in }\Omega . & 
\end{array}%
\right.  \label{abs-par-intro}
\end{equation}%
where $A$ is a given linear operator and $f$ is a given nonlinear map which
depends on the unknown PDE solution $u$. {Some examples of $f$ where our
theory directly applies are the following.}

\begin{itemize}
\item \textbf{Allen-Cahn (phase field) equation:} Here $f$ is a cubic type
of nonlinearity $f = -F^{\prime}$ associated with the double-well potential 
\begin{equation*}
F(u) = c_1 u^4 - c_2 u^2, \quad \mbox{for } c_2 > c_1 > 0 .
\end{equation*}

\item \textbf{Subdiffusive Fisher-KPP:} Here $f$ is a logistic term of the
form $r u (1-uK^{-1})$ with $r , K > 0$.

\item \textbf{Subdiffusive Burger's equation:} Here $f(u) := - u \mbox{div}%
(J*u)$, where 
\begin{equation*}
(J*u)(x) = \int_\Omega J(x-y) u(y) dy , \quad x \in \Omega .
\end{equation*}
\end{itemize}

We call $u$ the state variable. The control $z$ enters into the problem in a
linear fashion, but with the help of operator $\mathbb{B}$, it can be in the
interior, on the boundary, or in the exterior. Notice, that this framework
not only allows $A$ to be a standard local operator such as $-\mbox{div}(K
\nabla \cdot )$, but also a nonlocal operator such as a fractional Laplacian $%
(-\Delta)^s$ ($0<s<1$). In addition, the boundary control not only can, be of
Dirichlet, Neumann or Robin type, but also of Wentzell type \cite%
{MWarma_2006a}. Notice, that in the standard case of $\gamma = 1$, there are
several existing results on boundary control of Dirichlet, Neumann or Robin
problems, but there are no existing results on Wentzell type boundary
control under the weaker conditions provided in this paper.

A key novelty of this paper is the presence of fractional in time derivative 
$\partial _{t}^{\gamma }$ in \eqref{abs-par-intro} in the sense of Caputo
(Definition \ref{Caputo-der}). Recently, there has been a considerable
interest in optimal control of fractional PDEs and ODEs, but most of the
results are limited to linear problems or they consider very special
scenarios \cite{OPAgrawal_2008a,HAntil_EOtarola_AJSalgado_2015a,MR4050544}.
Besides, low regularity requirements, the interest in fractional time
derivative stems from its ability in capturing hereditary effects in
materials and anomalous random walks \cite%
{EBarkai_RMetzler_JKlafter_2000a,RMetzler_JKlafter_2004a}. This hereditary
property has also been recently used in designing new deep neural networks 
\cite{HAntil_HCElman_AOnwunta_DVerma_2021a,
HAntil_RKhatri_RLohner_DVerma_2020a}, and gradient based algorithms \cite%
{shin2021caputo}.

The amount of literature on fractional in time derivative is growing due to
many new emerging applications, however many of these results are empirical
observations, and the analytical results are still very limited. It is
well-known that the results from the classical models, i.e., $\gamma =1$, do
not directly extend to the fractional case $\gamma <1$, see the monograph 
\cite{GW} and references therein. However, the results in this monograph do
not apply to our problem as it focuses on the homogeneous case, i.e., $z=0$
and does not consider any optimal control problems. Furthermore, problem (%
\ref{abs-par-intro}) is \emph{ill-posed} (even when $f\equiv 0$) in the
sense that the solution flow is \textbf{not} strongly continuous (near $t=0$%
) in the scale of Hilbert spaces associated with \textbf{any} fractional
order of the (self-adjoint) operator $A$ (see Remark \ref{theta}). Another
technical difficulty is the fact that the formulation of the necessary
optimality conditions for the optimal control problem (\ref{theta}) requires
a rigorous passage of the Caputo-derivative from the forward to the backward
variables, a step which must be carefully analized (and, unfortunately,
omitted quite often in the current literature on optimal control for
subdiffusive parabolic problems). It turns out that the arguments, leading
to the neccesary optimality conditions, need an additional condition of the
behavior of the time derivative of the control variable near the origin. We
emphasize that this is a technical condition which is in fact necessary, and
which we believe, it cannot be discarded as it takes into account the
solution behavior of (\ref{abs-par-intro}) near the time $t=0$. Besides, the
control-to-state mapping for our problem appears to be (Fréchet)
differentiable \textbf{only} under this hypothesis.

The key novelties of this paper are the following.

\begin{itemize}
\item Well posedness of the fractional in time state equation under minimal (%
\textbf{generic})\ assumptions on $f$.

\item Well posedness of the linearized state (adjoint) equation.

\item Existence of solution to the optimal control problem.

\item Lipschitz continuity of the control to state map and its derivative.

\item Rigorous derivation of the first order necessary conditions using the
notion of \textbf{strong}\footnote{%
It remains unclear how this derivation can be performed with a notion of
generalized Caputo derivative (\ref{all-fract}), \textbf{without} the
validity of Proposition \ref{P1}.}\textbf{\ Caputo derivative }(\ref{frac-C}%
).

\item Applications to (subdiffusive) phase-transition phenomena,
(subdiffusive) Fisher-KPP equations and subdiffusive Burger's equation,
subject to nonlocal transport.
\end{itemize}

The remainder of the paper is organized as follows: In section~\ref{sec-2}
we introduce the basic definitions of fractional derivatives, properties of
abstract operator $A$ and collect various other tools that will be needed in
the remainder of the paper. Section~\ref{ss:sem} focuses on the
well-posedness of the forward semilinear problem \eqref{abs-par-intro} under
minimal conditions on $f$ and $z$. This is followed by section~\ref%
{ss:optimal} where we state the optimal control problem. Here we assume that
the final time $T$ is finite and $T<T_{\max },$ for a maximal time $T_{\max
}>0$ where either $T_{\max }=\infty $ (global solution) or $T_{\max }<\infty
,$ and finite time blow-up may occur in some $V_{\alpha }$-norm. Global
well-posedness is briefly touched upon in Section \ref{ss:global} for $V_{1}$%
-solutions. With respect to the optimal control problem, we first establish
the Lipschitz continuity of the control-to-state map and then existence of
solution to the optimal control problem. Next, we show the differentiability
of the control-to-state operator followed by Lipschitz continuity of the
derivative of the control-to-state map. Subsequently, we establish
well-posedness of the linearized state equation and rigorously derive the
first order necessary conditions.  We present several examples of
cost functionals and control problems in section~\ref{ss:csoe} where the
abstract theory can be applied.  Finally,  we provide in Section Appendix \ref{sec:ap} the proofs of most of teh technical results stated in Section \ref{ss:sem}.


\section{The basic functional framework}


\label{sec-2}

Let $Y,Z$ be two Banach spaces endowed with norms $\left\Vert \cdot
\right\Vert _{Y}$ and $\left\Vert \cdot \right\Vert _{Z}$, respectively. We
denote by $Y\hookrightarrow Z$ if $Y\subseteq Z$ and there exists a constant 
$C>0$ such that $\left\Vert u\right\Vert _{Z}\leq C\left\Vert u\right\Vert
_{Y},$ for $u\in Y\subseteq Z.$ In particular, this means that the injection
of $Y$ into $Z$ is continuous. In addition, if $Y$ is dense in $Z$, then we
denote by $Y\overset{d}{\hookrightarrow }Z$, and finally if the injection is
also compact we shall denote it by $Y\overset{c}{\hookrightarrow }Z$. We
denote by $\mathcal{L}(Y,Z)$ the space of all (bounded) linear operators
from $Y$ to $Z$. If $Y=Z$, we let $\mathcal{L}(Y,Z)=\mathcal{L}(Y)$. By the
dual $Y^{\ast }$\ of $Y$, we think of $Y^{\ast }$ as the set of all
(continuous) linear functionals on $Y$. When equipped with the operator norm 
$\left\Vert \cdot \right\Vert _{Y^{\ast }}$, $Y^{\ast }$ is also a Banach
space. We use throughout the notation $h\lesssim g$ to denote $h\leq Cg,$
for some constant $C>0$ when the dependance of the constant $C=C\left(
\gamma ,s,q,p,...\right) $ on some physical parameters is not relevant, and
so it is suppressed.

We give next the notion of fractional-in-time derivative in the sense of
Caputo and Riemann-Liouville. Let $\gamma \in \left( 0,1\right) $ and define%
\begin{equation*}
g_{\gamma }\left( t\right) :=%
\begin{cases}
\displaystyle\frac{t^{\gamma -1}}{\Gamma (\gamma )} & \text{if }t>0, \\ 
0 & \text{if }t\leq 0,%
\end{cases}%
\end{equation*}%
where $\Gamma $ is the usual Gamma function. Let $Y$ be a Banach space which
possesses the Radon-Nikodym property, and let $T>0.$

\begin{definition}[\textbf{{\textsf{Strong Caputo fractional derivative}}}]
\label{Caputo-der}Let $u\in W^{1,1}\left( (0,T);Y\right) .$ The (strong)
Caputo fractional derivative of order $\gamma \in \left( 0,1\right) $ is
given by%
\begin{equation}
_{C}\partial _{t}^{\gamma }u\left( t\right) :=\int_{0}^{t}g_{1-\gamma
}\left( t-\tau \right) u^{^{\prime }}\left( \tau \right) d\tau =(g_{1-\gamma
}\ast u^{^{\prime }})\left( t\right) ,  \label{frac-C}
\end{equation}%
for all $t\in (0,T].$
\end{definition}

\begin{definition}[\textbf{{\textsf{(Left) Riemann-Liouville fractional
derivative}}}]
Let $u\in C\left( [0,T];Y\right) $ be such that $g_{1-\gamma }\ast u\in
W^{1,1}\left( (0,T);Y\right) $. The (left)\ Riemann-Liouville fractional
derivative of order $\gamma \in \left( 0,1\right) $ is given by 
\begin{equation*}
D_{t}^{\gamma }u\left( t\right) :=\frac{d}{dt}\Big(g_{1-\gamma }\ast u\Big)%
(t)=\frac{d}{dt}\int_{0}^{t}g_{1-\gamma }\left( t-\tau \right) u\left( \tau
\right) d\tau ,
\end{equation*}%
for almost all $t\in (0,T).$
\end{definition}

We next set%
\begin{equation}
\partial _{t}^{\gamma }u:=D_{t}^{\gamma }\left( u-u\left( 0\right) \right) ,
\label{all-fract}
\end{equation}%
and recall that the right-hand side of (\ref{all-fract}) is usually dubbed
in the literature as a \textbf{generalized Caputo derivative }(see, e.g., 
\cite{GW})\textbf{. }We observe that the notions of fractional derivatives
in (\ref{all-fract}) and (\ref{frac-C}), respectively, are in fact
equivalent under the Radon-Nikodym property.

\begin{proposition}
\label{P1}Let the assumptions of Definition \ref{Caputo-der} be satisfied.
Then%
\begin{equation*}
_{C}\partial _{t}^{\gamma }u\left( t\right) =\partial _{t}(g_{1-\gamma }\ast
\left( u-u\left( 0\right) \right) )\left( t\right) =D_{t}^{\gamma }\left(
u-u\left( 0\right) \right) \left( t\right) ,
\end{equation*}%
for almost all $t\in (0,T],$ and $g_{1-\gamma }\ast (u-u\left( 0\right) )\in
W^{1,1}\left( (0,T);Y\right) .$
\end{proposition}

\begin{proof}
We note that $\left\Vert u\left( 0,\cdot\right) \right\Vert _{Y}<\infty $
since each $u\in W^{1,1}\left( (0,T);Y\right) $ is also continuous on $\left[
0,T\right] $ with values in $Y.$ As a consequence of the Radon-Nikodym
theorem (see \cite{DU77}), each function that belongs to $W^{1,1}\left(
[0,T];Y\right) $ is also absolutely continuous on $\left[ 0,T\right] $
(modulo a null set of Lebesgue measure) with values in $Y$. Then the
conclusion follows from a standard result in \cite[Theorem 3.1]{D} in the
case $Y=\mathbb{R}$ (the proof in the case of a general Banach space $Y$
follows with some, albeit, obvious modifications).
\end{proof}

\begin{definition}[\textbf{{\textsf{(Right) Riemann-Liouville fractional
derivative}}}]
The (right) Riemann-Liouville fractional derivative of order $\gamma \in
\left( 0,1\right) $ is defined by%
\begin{equation}
\partial _{t,T}^{\gamma }u(t)=-\frac{d}{dt}(I_{t,T}^{1-\gamma }u)(t),
\label{D-r}
\end{equation}%
where%
\begin{equation*}
I_{t,T}^{\gamma }u\left( t\right) :=\frac{1}{\Gamma (\gamma )}%
\int_{t}^{T}(\tau -t)^{\gamma -1}u(\tau )d\tau
\end{equation*}%
is the right Riemann-Liouville fractional integral of order $\gamma $. The
left Riemann-Liouville fractional integral of order $\gamma \in \left(
0,1\right) $ is given by 
\begin{equation*}
I_{0,t}^{\gamma }u(t)=\frac{1}{\Gamma (\gamma )}\int_{0}^{t}(t-\tau
)^{\gamma -1}u(\tau )d\tau .
\end{equation*}
\end{definition}

From (\ref{D-r}), we observe that if $u$ is differentiable, then $\partial
_{t,T}^{1}u=-\partial _{t}u$. We will employ the following well-known result
(see, e.g., \cite{Agr})\ to determine the corresponding adjoint problem
associated with the initial boundary value problem, that we have set up in
Section \ref{ss:optimal}.

\begin{proposition}
\label{ibpf}Under the assumptions of Proposition \ref{P1}, the following 
\emph{integration by parts formula} holds:%
\begin{equation}
\int_{0}^{T}v\left( t\right) _{C}\partial _{t}^{\gamma
}u(t)dt=\int_{0}^{T}\partial _{t,T}^{\gamma }v(t)u(t)dt+\left[
(I_{t,T}^{1-\gamma }v)\left( t\right) u(t)\right] _{t=0}^{t=T},
\label{bpf-eq}
\end{equation}%
provided that the left and right-hand sides expressions make sense (i.e.,
the products inside the integrals are well-defined, at least in a duality
sense $\left\langle \cdot ,\cdot \right\rangle _{Y,Y^{\ast }}$). Note that
if $u\left( 0\right) =(I_{t,T}^{1-\gamma }v)\left( T\right) =0,$ the bracket
term in (\ref{bpf-eq}) vanishes.
\end{proposition}

\begin{remark}
\emph{The integral expressions in (\ref{bpf-eq}) are well defined (cf. \cite[%
pg. 76]{SKM}, owing to $\partial _{t,T}^{\gamma }v=I_{t,T}^{1-\gamma
}(\partial _{t}v)$ and $_{C}\partial _{t}^{\gamma }u=I_{0,t}^{1-\gamma
}\left( \partial _{t}u\right) $), for instance, when $u\in W^{1,p}\left(
(0,T);Y^{\ast }\right) $ and $v\in W^{1,q}\left(( 0,T);Y\right) ,$ with $%
p,q\geq 1$ and $p^{-1}+q^{-1}\leq 2-\gamma $ ($p\neq 1$ and $q\neq 1$ when $%
p^{-1}+q^{-1}=2-\gamma $). }
\end{remark}

We recall next the following Gronwall type inequality from \cite[Lemma 6.3]%
{EL}.

\begin{lemma}
\label{A2}Let the function $\varphi (t)\geq 0$ be continuous for $0<t\leq T$%
. If 
\begin{equation*}
\varphi \left( t\right) \leq C_{1}t^{\alpha -1}+C_{2}\int_{0}^{t}\left(
t-s\right) ^{\beta -1}\varphi \left( s\right) ds,\text{ }0<t\leq T,
\end{equation*}%
for some constants $C_{1},C_{2}\geq 0$ and $\alpha ,\beta >0,$ then there is
a positive constant $C_{\ast }=C_{\ast }\left( \alpha ,\beta ,T,C_{2}\right) 
$ such that\footnote{%
A crucial part of this estimate that we shall exploit repeteadly is that the
final exponent in (\ref{gl}) is independent of $\beta .$}%
\begin{equation}
\varphi \left( t\right) \leq C_{\ast }C_{1}t^{\alpha -1},\text{ }0<t\leq T.
\label{gl}
\end{equation}
\end{lemma}

We shall now assume $\left( \Omega ,g\right) $ is a ($n$-dimensional)
Riemannian manifold and $g$ a complete Riemannian metric on $\Omega $, which
is at least of Lipschitz class (i.e., $\left\Vert g\right\Vert _{\text{Lip}%
}<\infty $).

\begin{enumerate}
\item[(HA)] In that case, let $A$ be a strictly positive\footnote{%
The strict positivity of the operator is generally not required, and one can
assume instead that $A\geq 0.$ Indeed, one can replace $A$ by $I+A$ by
adding the identity on both sides of equation (\ref{abs-par}). The
assumptions on $f$ hold for the modified nonlinearity as well.} self-adjoint
(unbounded) operator in $L^{2}(\Omega )$ (i.e., $0\in \rho \left( A\right) $%
) whose resolvent $\left( I+A\right) ^{-1}$ is compact\footnote{$L^{2}\left(
\Omega \right) $ is a locally compact since $\Omega $ is complete.} in $%
L^{2}\left( \Omega \right) $.
\end{enumerate}

By the spectral theory, $A$ has its eigenvalues forming a non-decreasing
sequence of real numbers $0<\lambda _{1}\leq \lambda _{2}\leq \cdots \leq
\lambda _{n}\leq \cdots $ satisfying $\lim_{n\rightarrow \infty }\lambda
_{n}=\infty $. In addition, the eigenvalues are of finite multiplicity. Let $%
(\varphi _{n})_{n\in {\mathbb{N}}}$ be the orthonormal basis of
eigenfunctions associated with $(\lambda _{n})_{n\in {\mathbb{N}}}$. Then $%
\varphi _{n}\in D(A)$ for every $n\in {\mathbb{N}}$, $(\varphi _{n})_{n\in {%
\mathbb{N}}}$ satisfies $A\varphi _{n}=\lambda _{n}\varphi _{n}$. We shall
denote by $D(A)^{\star }$ the dual of $D(A)$ with respect to the pivot space 
$L^{2}(\Omega )$ so that we have the continuous embedding $%
D(A)\overset{c}{\hookrightarrow } L^{2}(\Omega )\overset{c}{\hookrightarrow }D(A)^{\star }$. Spectral
theory also allows us to define the spaces $V_{\alpha }=D(A^{\alpha /2})$
with norms 
\begin{equation}
\left\vert u\right\vert _{\alpha }:=||A^{\alpha /2}u||_{L^{2}\left( \Omega
\right) },\text{ for any real }\alpha .  \label{fractional-p0}
\end{equation}%
A simple argument shows that \thinspace $D(A^{-\alpha /2})=V_{-\alpha },$
for $\alpha \geq 0$. Indeed, by the standard spectral theory, fractional
powers of $A$ can be defined by means of%
\begin{equation}  \label{fractional-p}
\begin{cases}
\displaystyle A^{\alpha /2}u & =\displaystyle\sum_{n=1}^{\infty }\lambda
_{n}^{\alpha /2}\left( u,\varphi _{n}\right) _{L^{2}(\Omega)}\varphi _{n},%
\text{ } \\ 
\displaystyle D(A^{\alpha /2}) & =\left\{ u\in L^{2}\left( \Omega \right)
:\left\Vert A^{\alpha /2}u\right\Vert _{L^{2}\left( \Omega \right) }^{2}=%
\displaystyle\sum_{n=1}^{\infty }\lambda _{n}^{\alpha }\left\vert \left(
u,\varphi _{n}\right) _{L^{2}(\Omega)}\right\vert ^{2}<\infty \right\} .%
\end{cases}%
\end{equation}%
In particular, there holds $V_{1}=D(A^{1/2})$ and $V_{-1}=V_{1}^{\ast }$
such that $\left\vert u\right\vert _{1}\simeq \left(
A^{1/2}u,A^{1/2}u\right) _{L^{2}(\Omega)}$ and $\left\vert u\right\vert
_{-1}\simeq \left( A^{-1}u,u\right) _{L^{2}}$ (in the sense of equivalent
norms), respectively. A complete characterization of the spaces $V_{\alpha }$
and their embedding properties into $L^{p}\left( \Omega \right) $-spaces ($%
p\geq 1$) can be found in \cite{AGW} (see also \cite{GM}).

Next, we recall the definition of the Wright type (also sometimes called the
Mainardi)\ function (see \cite{GW} and the references therein),%
\begin{equation}
\Phi _{\gamma }(z):=\sum_{n=0}^{\infty }\frac{(-z)^{n}}{n!\Gamma (-\gamma
n+1-\gamma )},\;\;0<\gamma <1,\;\;z\in \mathbb{C}.  \label{wright}
\end{equation}%
It is well known that $\Phi _{\gamma }(t)$ is a probability density
function, namely,%
\begin{equation*}
\Phi _{\gamma }(t)\geq 0,\quad t>0;\quad \int_{0}^{\infty }\Phi _{\gamma
}(t)dt=1.
\end{equation*}%
Furthermore, $\Phi _{\gamma }(0)=1/\Gamma \left( 1-\gamma \right) $ and%
\begin{equation}
\int_{0}^{\infty }t^{p}\Phi _{\gamma }(t)dt=\frac{\Gamma (p+1)}{\Gamma
(\gamma p+1)},\;\;p>-1,\;\;0<\gamma <1.  \label{moment}
\end{equation}

We let $\left( T\left( t\right) \right) _{t\geq 0}$ denote the analytic
semigroup on $L^{2}(\Omega )$ generated by the operator $-A$, and consider
an extension of $T$ (which we still denote by $T$, for the simplicity of
notation) on all scales of negative fractional order spaces $V_{\alpha },$ $%
\alpha <0$. Next, we define two additional operators%
\begin{equation*}
S_{\gamma }(t):V_{\alpha }\rightarrow V_{\alpha },\text{ }P_{\gamma
}(t):V_{\alpha }\rightarrow V_{\alpha },\text{ }\alpha \in \left[ -2,2\right]
,
\end{equation*}%
by%
\begin{equation}
\begin{cases}
\displaystyle S_{\gamma }(t)v:=\int_{0}^{\infty }\Phi _{\gamma }(\tau
)T(\tau t^{\gamma })vd\tau , \\ 
\displaystyle P_{\gamma }(t)v:=\gamma t^{\gamma -1}\int_{0}^{\infty }\tau
\Phi _{\gamma }(\tau )T(\tau t^{\gamma })vd\tau .%
\end{cases}
\label{op-SP}
\end{equation}

We next recall the definition of the Mittag-Leffler function%
\begin{equation*}
E_{\alpha ,\beta }(z):=\sum_{n=0}^{\infty }\frac{z^{n}}{\Gamma (\alpha
n+\beta )},\;\;\alpha >0,\;\beta \in {\mathbb{C}},\quad z\in {\mathbb{C}}.
\end{equation*}%
A particular estimate for $E_{\alpha ,\beta }$ that we shall often use in
the paper (and whenever necessary) is%
\begin{equation}
\left\vert E_{\alpha ,\beta }\left( x\right) \right\vert \leq C_{\alpha
,\beta }\left( 1-x\right) ^{-1},\text{ \ for }x\leq 0\text{ and }\alpha \in
\left( 0,1\right) ,\text{ }\beta >0.  \label{ML-bound}
\end{equation}%
It is also well-known that $E_{\alpha ,\beta }(z)$ is an entire function,
and that both operators in (\ref{op-SP}) can be also cast in terms of these
functions (see, for instance, \cite[Theorem 4.2]{MN}). In particular, by the
spectral theory we further have%
\begin{equation}
\left\{ 
\begin{array}{ll}
\displaystyle S_{\gamma }(t)v=\sum_{n=0}^{\infty }\left( v,\varphi
_{n}\right) _{L^{2}(\Omega)}E_{\gamma ,1}\left( -\lambda _{n}t^{\gamma
}\right) \varphi _{n}, &  \\ 
&  \\ 
\displaystyle P_{\gamma }(t)v=\sum_{n=0}^{\infty }\left( v,\varphi
_{n}\right) _{L^{2}(\Omega)}t^{\gamma -1}E_{\gamma ,\gamma }\left( -\lambda
_{n}t^{\gamma }\right) \varphi _{n}. & 
\end{array}%
\right.  \label{op-SP2}
\end{equation}

\begin{proposition}
\label{est-SP}The operator families $\left\{ S_{\gamma }\left( t\right)
\right\} ,$ $\left\{ P_{\gamma }\left( t\right) \right\} $ are analytic for $%
t>0,$ and satisfy the following estimates:%
\begin{equation}
\left\vert S_{\gamma }\left( t\right) v\right\vert _{\beta }\leq C_{\alpha
,\beta ,\gamma }t^{-\gamma \left( \frac{\beta -\alpha }{2}\right)
}\left\vert v\right\vert _{\alpha },\text{ }-2\leq \alpha \leq \beta \leq 2,
\label{s-1}
\end{equation}%
with $\beta -\alpha \in \lbrack 0,2],$ and%
\begin{equation}
\left\vert P_{\gamma }\left( t\right) v\right\vert _{\widetilde{\beta }}\leq
C_{\widetilde{\alpha },\widetilde{\beta },\gamma }t^{\gamma -1-\gamma \left( 
\frac{\widetilde{\beta }-\widetilde{\alpha }}{2}\right) }\left\vert
v\right\vert _{\widetilde{\alpha }},\text{ }-2\leq \widetilde{\alpha }\leq 
\widetilde{\beta }\leq 2.  \label{s-2}
\end{equation}%
The positive constants $C_{\alpha ,\beta ,\gamma },C_{\widetilde{\alpha },%
\widetilde{\beta },\gamma }$ are independent of $t$ and $v,$ and are bounded
as $\gamma \rightarrow 1^{-}.$ Finally, the operator $S_{\gamma }$ is also a
contraction (strongly continuous)\ mapping from $V_{\alpha }\rightarrow
V_{\alpha }.$
\end{proposition}

\begin{proof}
The analyticity of the semigroup $T\left( t\right) =\exp \left( -tA\right) $
together with the representation \eqref{op-SP} implies the analyticity of $%
S_\gamma(t)$ and $P_\gamma(t)$ for $t>0$. The analyticity of $T\left(
t\right)$ is also reflected in the inequality%
\begin{equation}
\left\vert T\left( t\right) v\right\vert _{\beta }\leq C_{\alpha ,\beta
}t^{-\left( \beta -\alpha \right) /2}\left\vert v\right\vert _{\alpha },%
\text{ with }-2\leq \alpha \leq \beta \leq 2.  \label{s-sp}
\end{equation}%
Combining (\ref{s-sp}) with the norm $\left\vert v\right\vert _{\alpha }$
via (\ref{fractional-p0})-(\ref{fractional-p}) for all real $\alpha $, and
exploiting the identities (\ref{moment})-(\ref{op-SP}) (or, respectively (%
\ref{ML-bound})-(\ref{op-SP2})), we easily obtain the estimates (\ref{s-1})-(%
\ref{s-2}).
\end{proof}

\section{Well-posedness results for the forward problem}

\label{ss:sem}

In the above framework, we can conveniently rewrite the semilinear problem
as follows:%
\begin{equation}
\left\{ 
\begin{array}{ll}
\partial _{t}^{\gamma }u\left( x,t\right) +Au\left( x,t\right) =f\left(
u\left( x,t\right) \right) +\mathbb{B}z\left( x,t\right) , & \text{in }%
\Omega \times (0,\infty ), \\ 
u\left( \cdot ,0\right) =u_{0}\text{ in }\Omega . & 
\end{array}%
\right.  \label{abs-par}
\end{equation}%
Our main goal in this section is to state sufficiently general conditions on
the data $\left( f,z,u_{0}\right) $ for which we can infer the existence of
properly-defined solutions\footnote{%
We will also investigate in which sense (\ref{abs-par}) is satisfied by
non-regular solutions.} for (\ref{abs-par}). Let $T\in \left( 0,\infty
\right) $ and denote by $J$ a time interval of the form $\left[ 0,T\right] ,$
$[0,T)$ or $[0,\infty )$.

\begin{definition}
\label{mild-sol}By a mild solution of (\ref{abs-par}) on the interval $J$,
we mean that the measurable function $u$ has the following properties:

\begin{enumerate}
\item $u\in C\left( J;V_{\alpha }\right) ,$ for some $\alpha \in \mathbb{R}$.

\item $f\left( u\left( \cdot ,t\right) \right) \in V_{\beta },$ for all $%
t\in J,$ for some $\alpha \geq \beta ,$ $\beta \in \mathbb{R}$.

\item $\displaystyle u\left( \cdot ,t\right) =S_{\gamma }\left( t\right)
u_{0}+\int_{0}^{t}P_{\gamma }\left( t-\tau \right) f\left( u\left( \cdot
,\tau \right) \right) d\tau +\int_{0}^{t}P_{\gamma }\left( t-\tau \right) 
\mathbb{B}z\left( \cdot ,\tau \right) d\tau ,$ for all $t\in J\backslash
\left\{ 0\right\} ,$ where the integral is an absolutely converging Bochner
integral in the space $V_{\alpha }.$

\item The initial datum $u_{0}$ is assumed in the following sense:%
\begin{equation}
\lim_{t\downarrow 0^{+}}\left\vert u\left( \cdot ,t\right) -u_{0}\right\vert
_{\alpha }=0,  \label{ini-cont}
\end{equation}%
for $u_{0}\in V_{\alpha }.$
\end{enumerate}
\end{definition}

Our first goal is to establish the existence of maximally-defined mild
solutions under some suitable assumptions on the parameters of the problem.
These assumptions are as follows.

\begin{enumerate}
\item[(H1)] Let $\widetilde{\alpha }\in \lbrack -1,0]$ be such that $\mathbb{%
B}\in \mathcal{L}\left( L^{2}(D);V_{\widetilde{\alpha }}\right) $, for an
arbitrary\footnote{%
The control region $D$ depends upon the choice of $A,$ the geometry of $%
\Omega $ and the control operator $\mathbb{B}$, as some examples will show
in Section \ref{ss:csoe}.} (Hausdorff, at least) space $D$. Set $I_{\beta
}:=[\beta ,\beta +2),$ for $\beta \in \mathbb{R}$ and assume that $u_{0}\in
V_{\alpha },$ for some $\alpha \in I_{\widetilde{\alpha }}\cap I_{\beta
}\neq \varnothing .$

\item[(H2)] The 'control' function $z\in L_{\text{loc}}^{q}(\mathbb{R}%
_{+};L^{2}\mathbb{(}D\mathbb{)}),$ for some $q\in (\frac{2}{\gamma \left(
2-\alpha +\widetilde{\alpha }\right) },\infty ].$

\item[(H3)] The nonlinear function $f\in C_{\text{loc}}^{0,1}\left( \mathbb{R%
}\right) ,$ $f\left( 0\right) =0,$ induces a locally Lipschitzian map 
\begin{equation*}
f:V_{\alpha }\rightarrow V_{\beta };
\end{equation*}%
namely, for every $R>0$, with $\left\vert u\right\vert _{\alpha },\left\vert
v\right\vert _{\alpha }\leq R,$ there exists $C_{R}>0$ such that%
\begin{equation*}
\left\vert f\left( u\right) -f\left( v\right) \right\vert _{\beta }\leq
C_{R}\left\vert u-v\right\vert _{\alpha },\text{ for all }u,v\in V_{\alpha
}\hookrightarrow V_{\beta }.
\end{equation*}
\end{enumerate}

Our first result is concerned with the existence and uniqueness of
locally-defined mild solutions (see Appendix \ref{sec:ap} for the proof).

\begin{lemma}
\label{T1-integral-H1}(\textbf{Local existence}). Assume (HA) and (H1)-(H3).
Then there exists a time $T_{\ast }>0$ (depending on $u_{0}$) such that the
problem (\ref{abs-par}) possesses a unique mild solution in the sense of
Definition \ref{mild-sol} on the interval $J=\left[ 0,T_{\ast }\right] .$
\end{lemma}

Our second statement shows that locally-defined mild solutions in $V_{\alpha
}$ can be (uniquely) extended to a larger interval (see Appendix \ref{sec:ap}
for the proof).

\begin{lemma}
\label{extension} (\textbf{Unique continuation}) Let the assumptions of
Lemma \ref{T1-integral-H1} be satisfied. Then the unique integral solution
on $J=[0,T^{\star }]$ of (\ref{abs-par}) can be extended to the interval $%
[0,T^{\star }+\tau ]$, for some $\tau >0$, so that, the extended function is
the unique mild solution of (\ref{abs-par}) on $[0,T^{\star }+\tau ]$ in the
sense of Definition \ref{mild-sol}.
\end{lemma}

The following statement is then straightforward on account of the above
lemmas.

\begin{theorem}
\label{prelim}Assume (HA) and (H1)-(H3). Problem (\ref{abs-par}) has a
unique mild solution on $J=[0,T_{\max })$ in the sense of Definition \ref%
{mild-sol}, where either $T_{\max }=\infty $ or $T_{\max }<\infty $, and in
that case, 
\begin{equation*}
\underset{t\rightarrow T_{\max }^{-}}{\lim \sup }\left\vert u\left( t\right)
\right\vert _{\alpha }=\infty .
\end{equation*}
\end{theorem}

\begin{proof}
The proof is standard owing to Lemma \ref{T1-integral-H1}\ and Lemma \ref%
{extension}, respectively, and a contradiction argument (see, e.g., \cite%
{AGW,MN}).
\end{proof}

We derive a sufficient condition in order to conclude additional (temporal)
regularity for the mild solution (under the same assumptions of Theorem \ref%
{prelim}; see Appendix \ref{sec:ap}, for a proof).

\begin{theorem}
\label{reg-thm}Let $u:[0,T]\rightarrow V_{\alpha }$ be a mild solution in
the sense of Theorem \ref{prelim} for any $T<T_{\max }$. In addition, assume
that $u_{0}\in V_{\beta +2}\hookrightarrow V_{\alpha }$ and $z\in
W^{1,1}(\left(0,T\right) ;L^{2}\left( D\right) ),$ provided that%
\begin{equation}
\left\Vert \partial _{t}z\left( t\right) \right\Vert _{L^{2}\left( D\right)
}\leq Ct^{\rho -1},\text{ \ for all }0<t\leq T,  \label{z-reg}
\end{equation}%
for some $\rho >0,$ and $C>0$ independent of $t$ and $z$. Then, for the
above mild solution, we have 
\begin{equation}
u\in W^{1,1+\xi }\left( \left( 0,T\right) ;V_{\alpha }\right) ,\text{ for
some }\xi >0.  \label{u-reg}
\end{equation}%
Furthermore, the identity $_{C}\partial _{t}^{\gamma }u\left( t\right)
=\partial _{t}^{\gamma }u\left( t\right) $ is satisfied for almost all $t\in
\left( 0,T\right) .$
\end{theorem}

In what follows it is more convenient to set $\widetilde{\alpha }=\beta \in
\lbrack -1,0]$ and recall that $\alpha \in \lbrack \beta ,\beta +2)$ (so
that assumption (H1) is satisfied). The previous theorems imply the
following two (major) statements which conclude the section.

\begin{theorem}
\label{main1}(\textbf{The problem for the generalized Caputo derivative}) If
(HA), (H1)-(H3) hold, then problem (\ref{abs-par}) has a unique mild
solution $u\in C\left( [0,T_{\max });V_{\alpha }\right) ,$ for which%
\begin{equation}
\partial _{t}^{\gamma }u:=\partial _{t}\left( g_{1-\gamma }\ast \left(
u-u_{0}\right) \right) \in C([0,T_{\max });V_{\widetilde{\alpha }-\delta }),%
\text{ for any }\delta \in \left(\frac{2}{\gamma q},2\right].
\label{time-reg1}
\end{equation}%
Moreover, the variational identity%
\begin{equation}
\langle \partial _{t}^{\gamma }u\left( t\right) +Au\left( t\right) ,v\rangle
_{V_{\widetilde{\alpha }-\delta },V_{-\widetilde{\alpha }+\delta }}=\langle
f\left( u\left( t\right) \right) +\mathbb{B}z\left( t\right) ,v\rangle _{V_{%
\widetilde{\alpha }},V_{-\widetilde{\alpha }}},  \label{var}
\end{equation}%
holds for any $v\in V_{-\widetilde{\alpha }+\delta }\subset V_{-\widetilde{%
\alpha }}$, and for almost all $t\in (0,T_{\max }).$
\end{theorem}

\begin{proof}
The statement (\ref{time-reg1}) is a consequence of the proof of Theorem \ref%
{prelim}, since $f\left( u\right) \in C([0,T_{\max });V_{\widetilde{\alpha }%
})\subset L^{q}\left( 0,T_{\max };V_{\widetilde{\alpha }}\right) $, 
\begin{align*}
& \left[ \left\vert A\left( P_{\gamma }\ast \mathbb{B}z\right) \left(
t\right) \right\vert _{\widetilde{\alpha }-\delta }+\left\vert A\left(
P_{\gamma }\ast f\left( u\right) \right) \left( t\right) \right\vert _{%
\widetilde{\alpha }-\delta }\right] \\
& \lesssim t^{\frac{\gamma \delta }{2}-\frac{1}{q}}\left( \left\Vert \mathbb{%
B}z\right\Vert _{L^{q}\left( 0,T;V_{\widetilde{\alpha }}\right) }+\left\Vert
f\left( u\right) \right\Vert _{L^{q}\left( 0,T;V_{\widetilde{\alpha }%
}\right) }\right) ,
\end{align*}%
and%
\begin{equation*}
\left\vert AS_{\gamma }\left( t\right) u_{0}\right\vert _{\widetilde{\alpha }%
-\delta }\leq \left\vert AS_{\gamma }\left( t\right) u_{0}\right\vert _{%
\widetilde{\alpha }}\lesssim t^{\frac{\gamma }{2}\left( 2+\widetilde{\alpha }%
-\alpha \right) }\left\vert u_{0}\right\vert _{\alpha },
\end{equation*}%
for all $T_{\max }>t>0$. The identity (\ref{var}) then follows from the
solution representation in Definition \ref{mild-sol} and from (\ref%
{time-reg1}).
\end{proof}

\begin{corollary}
\label{main2}(\textbf{The problem for the strong Caputo derivative}) Let the
assumptions of Theorem \ref{main1} be satisfied, and in addition, assume $%
u_{0}\in V_{\beta +2}\subset V_{\alpha }$ and (\ref{z-reg}) hold. Then the
mild solution of problem (\ref{abs-par}) satisfies $\left\vert \partial
_{t}u\left( t\right) \right\vert _{\alpha }\lesssim t^{\theta -1}$, and
therefore,%
\begin{equation}
u\in W^{1,1+\xi }\left( \left(0,T\right) ;V_{\alpha }\right) \cap L^{\sigma
}\left( (0,T);V_{2+\widetilde{\alpha }}\right) ,\text{ }_{C}\partial
_{t}^{\gamma }u\in L^{1+\xi }\left((0,T);V_{\widetilde{\alpha }}\right) ,
\label{u-reg2}
\end{equation}%
for $\xi =\xi \left( \theta \right) \in (0,\frac{\theta }{1-\theta })$, $%
\sigma :=\min \left( 1+\xi ,q\right) >1$. Moreover, the variational identity%
\begin{equation}
\langle \partial _{t}^{\gamma }u\left( t\right) +Au\left( t\right) ,v\rangle
_{V_{\widetilde{\alpha }},V_{-\widetilde{\alpha }}}=\langle f\left( u\left(
t\right) \right) +\mathbb{B}z\left( t\right) ,v\rangle _{V_{\widetilde{%
\alpha }},V_{-\widetilde{\alpha }}},  \label{var2}
\end{equation}%
holds for any $v\in V_{-\widetilde{\alpha }}$, and for almost all $t\in
(0,T_{\max }).$
\end{corollary}

\begin{proof}
The proof is a consequence of Theorem \ref{reg-thm}, in view of the
assumptions of Theorem \ref{main1}. Indeed, owing to the fact that $\mathbb{B%
}z\in L^{q}\left( (0,T);V_{\widetilde{\alpha }}\right) $, $f\left( u\right)
\in C\left( \left[ 0,T\right] ;V_{\widetilde{\alpha }}\right) $ and $%
_{C}\partial _{t}^{\gamma }u=\partial _{t}^{\gamma }u\in L^{1+\xi }\left(
(0,T);V_{\widetilde{\alpha }}\right) $, one may argue, by comparison in (\ref%
{var}), that $Au\in L^{\sigma }\left( (0,T);V_{\widetilde{\alpha }}\right) ,$
which implies that $u\in L^{\sigma }\left( (0,T);V_{2+\widetilde{\alpha }%
}\right) $.
\end{proof}


\begin{remark}
\label{theta}\emph{We observe the following facts. }

\begin{enumerate}
\item \emph{When $\widetilde{\alpha }=\beta \in \lbrack -1,0]$ and $\alpha
\in \lbrack \beta ,\beta +2),$ we have in Corollary \ref{main2}, $\theta =%
\frac{\gamma \left( 2-\alpha +\beta \right) }{2}-\frac{1}{q}>0$ whenever $%
q\in (\frac{2}{\gamma \left( 2-\alpha +\beta \right) },\infty ].$ Moreover, $%
\sigma =\min \left( 1+\xi ,q\right) =1+\xi \in \left( 1,\eta \left( \theta
\right) \right) $ since $\eta \left( \theta \right) :=\frac{1}{1-\theta }\in
(1,q].$ }

\item \emph{When $\widetilde{\alpha }\neq \beta $, the explicit value of $%
\theta >0,$ in terms of $\alpha ,\beta ,\gamma $ and $\widetilde{\alpha },$
can be found in (\ref{thetabis}), namely\footnote{\emph{The value of $\theta 
$ is independent of $\rho >0.$}},%
\begin{equation*}
\theta =\min \left\{ \frac{\gamma }{2}\left( 2-\alpha +\widetilde{\alpha }%
\right) -\frac{1}{q},\frac{\gamma }{2}\left( 2-\alpha +\beta \right)
\right\} .
\end{equation*}%
In this case, the analogue of Corollary \ref{main2} is%
\begin{equation}
\partial _{t}^{\gamma }u+Au=f\left( u\right) +\mathbb{B}z,\text{ holds in }%
V_{\min \left\{ \widetilde{\alpha },\beta \right\} },\text{ for almost all }%
t\in \left( 0,T\right) ,  \label{weaky}
\end{equation}%
and each solution of (\ref{weaky}) belongs to 
\begin{equation*}
W^{1,1+\xi }\left( \left(0,T\right) ;V_{\alpha }\right) \cap L^{\sigma
}\left( (0,T);V_{2+\min \left\{ \widetilde{\alpha },\beta \right\} }\right)
,_{C}\partial _{t}^{\gamma }u\in L^{1+\xi }\left( (0,T);V_{\min \left\{ 
\widetilde{\alpha },\beta \right\} }\right) .
\end{equation*}%
For additional regularity theory of problem (\ref{abs-par}), the restriction 
$\widetilde{\alpha }\in \left[ -1,0\right] $ does not appear to be a
necessary condition. More precisely, in the statement of Corollary \ref%
{main2}, one can assume instead that $\widetilde{\alpha }\in \mathbb{R}_{+}$%
, for as long as $\alpha \in I_{\widetilde{\alpha }}\cap I_{\beta }\neq
\varnothing $ (see, e.g., Theorem \ref{thm-reg} in Section \ref{ss:global}).
However, the restriction that $\widetilde{\alpha }>0$ appears necessary for
the rigourous justification of the $V_{1}$-energy equality for the
corresponding subdiffusive problem. }

\item \emph{The gap in regularity between the initial datum $u_{0}\in
V_{\beta +2}\subsetneqq V_{\alpha }$ and the solution $u\in C(\left[ 0,T%
\right] ;V_{\alpha })$ (but\footnote{\emph{Strictly speaking, this is a
consequence of the singular behavior of $P_{\gamma }\left( t\right) $ near $%
t=0.$}} with $u\notin C(\left[ 0,T\right] ;V_{\beta +2})$) is not of
technical nature. It is due to a complete \textbf{failure} of the semigroup
property for the solution operator associated with (\ref{abs-par}) (see,
e.g., \cite{GW, E2016}, and the references therein). In particular, this
means that the fractional in time problem is not well-posed in the classical
sense formulated by Hadamard; namely, there \textbf{does not exist a
(strongly) continuous} flow map $\Phi :u_{0}\mapsto u\left( t\right) $ in
any scale of the operator spaces $V_{\alpha }$. This is in contrast to the
strong continuity of the flow map for the classical problem when $\gamma =1.$
}
\end{enumerate}
\end{remark}

\section{The optimal control problem}

\label{ss:optimal}

In view of Corollary \ref{main2}, we set the control space to be%
\begin{equation*}
Z_{\rho ,\infty }:=\left\{ z\in C(\left[ 0,T\right] ;L^{2}(D)):\left\Vert
\partial _{t}z\left( t\right) \right\Vert _{L^{2}\left( D\right) }\lesssim
t^{\rho -1},\text{ \ a.e. }0<t\leq T\right\} ,
\end{equation*}%
for some $T<T_{\max }$ (with a value $T$ which we will fix from now on) and $%
0<\rho \leq 1$. Notice\footnote{%
Due to the embedding, we may immediately take $q=\infty $ in Remark \ref%
{theta}. It follows that $\sigma =1+\xi \left( \theta \right) >1$ and $%
\theta =\frac{\gamma \left( 2-\alpha +\beta \right) }{2}$.} that $Z_{\rho
,\infty }$ is a closed (bounded)\ subset of 
\begin{equation*}
W^{1,1+\lambda }(\left(0,T\right) ;L^{2}\left( D\right) )\subset C(\left[ 0,T%
\right] ;L^{2}\left( D\right) ),
\end{equation*}%
for some $\lambda >0$ depending on $\rho $ (i.e, $\left( 1+\lambda \right)
\left( \rho -1\right) >-1$), when we endow it with the norm%
\begin{equation}
\left\Vert z\right\Vert _{Z_{\rho ,\infty }}:=\left\Vert z\right\Vert
_{C([0,T];L^{2}\left( D\right) )}+\sup_{t\in \left[ 0,T\right] }t^{1-\rho
}\left\Vert \partial _{t}z\left( t\right) \right\Vert _{L^{2}\left( D\right)
}.  \label{norm-c}
\end{equation}%
Thus, the control-to-state operator%
\begin{equation*}
\mathcal{S}:Z_{\rho ,\infty }\rightarrow Y_{\theta ,\alpha },\text{ }%
z\longmapsto u=:\mathcal{S}\left( z\right)
\end{equation*}%
is well-defined as a mapping from $Z_{\rho ,\infty }$ into the Banach space%
\begin{equation}
Y_{\theta ,\alpha }=\left\{ u\in W^{1,1+\xi \left( \theta \right) }\left(
\left( 0,T\right) ;V_{\alpha }\right) :\left\vert \partial _{t}u\left(
t\right) \right\vert _{\alpha }\lesssim t^{\theta -1},\text{ a.e. }0<t\leq
T\right\} ;  \label{Y-space}
\end{equation}%
$Y_{\theta ,\alpha }$ is endowed with the natural norm (for some $\xi =\xi
\left( \theta \right) >0,$ depending\footnote{%
Namely, $\xi >0$ is such that $\left( 1+\xi \right) \left( \theta -1\right)
>-1.$} on $\theta >0$)%
\begin{equation*}
\left\Vert u\right\Vert _{Y_{\theta ,\alpha }}:=\sup_{t\in \left[ 0,T\right]
}\left\vert u\left( t\right) \right\vert _{\alpha }+\sup_{t\in \left[ 0,T%
\right] }t^{1-\theta }\left\vert \partial _{t}u\left( t\right) \right\vert
_{\alpha }.
\end{equation*}

We notice first that $\mathcal{S}$ is Lipschitz continuous from $Z_{\rho
,\infty }$ into $C\left( \left[ 0,T\right] ;V_{\alpha }\right) $ (see (\ref%
{LC1})). Secondly, if we impose additional assumptions on $f,$ the state
mapping $\mathcal{S}$ satisfies an improved stability estimate (\ref{LC2}).
To this end, let us denote by $\mathcal{U}$ a nonempty, open and bounded
subset of $Z_{\rho ,\infty }.$

\begin{enumerate}
\item[(H4)] The nonlinearity $f\in C^{1,1}$ induces a bounded\footnote{%
Note that boundedness is a consequence of (\ref{g-map}).} mapping%
\begin{equation}
g_{u}\left( v\right) :=\left( \partial _{u}f\left( u\right) \right)
v:V_{\alpha }\rightarrow V_{\beta },  \label{g-map2}
\end{equation}%
such that, for every $u_{1},u_{2}\in V_{\alpha }$ satisfying $\left\vert
u_{i}\right\vert _{\alpha }\leq R$ ($i=1,2$), there is a constant $C_{R}>0$
such that, 
\begin{equation}
\left\vert g_{u_{1}}\left( v\right) -g_{u_{2}}\left( v\right) \right\vert
_{\beta }\leq C_{R}\left\vert v\right\vert _{\alpha }\left\vert
u_{1}-u_{2}\right\vert _{\alpha }.  \label{g-map}
\end{equation}
\end{enumerate}

\begin{theorem}
\label{state-map}Let the assumptions of Corollary \ref{main2} be satisfied.

\begin{enumerate}
\item[(i)] Then for each $T<T_{\max },$ there exists a constant $K_{1}>0$,
depending on $T,R,f$, such that whenver $z_{1},z_{2}\in \mathcal{U}$ are
given and $u_{1},u_{2}\in Y_{\theta ,\alpha }$ denote the associated
solutions of the state system, we have%
\begin{equation}
\left\Vert u_{1}-u_{2}\right\Vert _{C\left( \left[ 0,T\right] ;V_{\alpha
}\right) }\leq K_{1}\left\Vert z_{1}-z_{2}\right\Vert _{C([0,T];L^{2}\left(
D\right) )}.  \label{LC1}
\end{equation}

\item[(ii)] If in addition, (H4) holds, then there is a constant $K_{2}>0$
such that%
\begin{equation}
\sup_{t\in \left[ 0,T\right] }t^{1-\theta }\left\vert \left( \partial
_{t}u_{1}-\partial _{t}u_{2}\right) \left( t\right) \right\vert _{\alpha
}\leq K_{2}\left\Vert z_{1}-z_{2}\right\Vert _{Z_{\rho ,\infty }}.
\label{LC2}
\end{equation}
\end{enumerate}
\end{theorem}

\begin{proof}
Set $u:=u_{1}-u_{2}$ and $z:=z_{1}-z_{2}.$ Then, every mild/weak solution $u$
satisfies on $\left( 0,T\right) \subset \left( 0,T_{\max }\right) ,$ 
\begin{equation}
u\left( t\right) =\int_{0}^{t}P_{\gamma }\left( t-\tau \right) \left(
f\left( u_{1}\left( \tau \right) \right) -f\left( u_{2}\left( \tau \right)
\right) \right) d\tau +\int_{0}^{t}P_{\gamma }\left( t-\tau \right) \mathbb{B%
}z\left( \tau \right) d\tau  \label{u-dif1}
\end{equation}%
since $u\left( 0\right) =0$, for $i=1,2,$ whenever $u_{i}\left( 0\right)
=u_{0}$. By Theorem \ref{reg-thm},%
\begin{align}
\partial _{t}u_{i}\left( t\right) & =AP_{\gamma }\left( t\right)
u_{0}+P_{\gamma }\left( t\right) f\left( u_{0}\right) +P_{\gamma }\left(
t\right) \mathbb{B}z_{i}\left( 0\right)  \label{u-dif2} \\
& +\int_{0}^{t}P_{\gamma }\left( t-\tau \right) \mathbb{B\partial }%
_{t}z_{i}\left( \tau \right) d\tau +\int_{0}^{t}P_{\gamma }\left( t-\tau
\right) \partial _{u_{i}}f\left( u_{i}\left( \tau \right) \right) \partial
_{t}u_{i}\left( \tau \right) d\tau ,  \notag
\end{align}%
for almost all $t\in (0,T)\subset \left( 0,T_{\max }\right) .$ Then (\ref%
{LC1}) follows easily employing once again an extension argument for (\ref%
{u-dif1}) to the whole interval $\left( 0,T\right) $, via the proofs of
Lemma \ref{T1-integral-H1} and Lemma \ref{extension}. Let us now set $%
v:=\partial _{t}u_{1}-\partial _{t}u_{2},$ and notice that%
\begin{align}
v\left( t\right) & =\int_{0}^{t}P_{\gamma }\left( t-\tau \right) \left(
g_{u_{1}}\left( \partial _{t}u_{1}\left( \tau \right) \right)
-g_{u_{2}}\left( \partial _{t}u_{1}\left( \tau \right) \right) \right) d\tau
+P_{\gamma }\left( t\right) \mathbb{B}z\left( 0\right)  \label{u-diff3} \\
& +\int_{0}^{t}P_{\gamma }\left( t-\tau \right) \partial _{u_{2}}f\left(
u_{2}\left( \tau \right) \right) v\left( \tau \right) d\tau
+\int_{0}^{t}P_{\gamma }\left( t-\tau \right) \mathbb{B\partial }_{t}z\left(
\tau \right) d\tau .  \notag
\end{align}%
The argument in (\ref{Z4bis}), given $\widetilde{\alpha }=\beta $ and $%
\alpha \in \lbrack \beta ,\beta +2),$ easily yields%
\begin{align*}
\left\vert \int_{0}^{t}P_{\gamma }\left( t-\tau \right) \mathbb{B\partial }%
_{t}z\left( \tau \right) d\tau \right\vert _{\alpha }& \lesssim t^{\frac{%
\gamma \left( 2-\alpha +\beta \right) }{2}+\rho -1}\sup_{t\in \left[ 0,T%
\right] }t^{1-\rho }\left\Vert \partial _{t}z\left( t\right) \right\Vert
_{L^{2}\left( D\right) } \\
& \lesssim t^{\frac{\gamma \left( 2-\alpha +\beta \right) }{2}+\rho
-1}\left\Vert z\right\Vert _{Z_{\rho ,\infty }} \\
& \lesssim t^{\theta -1}T^{\rho }\left\Vert z\right\Vert _{Z_{\rho ,\infty }}
\\
& \lesssim t^{\theta -1}\left\Vert z\right\Vert _{Z_{\rho ,\infty }}
\end{align*}%
while, in light of the boundedness of $g_{u}$ (by (\ref{g-map}), $\left\vert
g_{u}\left( v\right) \right\vert _{\beta }\lesssim \left\vert v\right\vert
_{\alpha }\left( \left\vert u\right\vert _{\alpha }+1\right) ,$ since $0\in
V_{\alpha }$), it follows that%
\begin{equation*}
\left\vert \int_{0}^{t}P_{\gamma }\left( t-\tau \right) \partial
_{u_{2}}f\left( u_{2}\left( \tau \right) \right) v\left( \tau \right) d\tau
\right\vert _{\alpha }\lesssim \int_{0}^{t}\left( t-\tau \right) ^{\frac{%
\gamma \left( 2-\alpha +\beta \right) }{2}-1}\left\vert v\left( \tau \right)
\right\vert _{\alpha }d\tau .
\end{equation*}%
Moreover, in view of (\ref{g-map}) and (\ref{LC1}), we can deduce that 
\begin{align*}
& \left\vert \int_{0}^{t}P_{\gamma }\left( t-\tau \right) \left( \partial
_{u_{1}}f\left( u_{1}\left( \tau \right) \right) -\partial _{u_{2}}f\left(
u_{2}\left( \tau \right) \right) \right) \partial _{t}u_{1}\left( \tau
\right) d\tau \right\vert _{\alpha } \\
& \lesssim \int_{0}^{t}\left( t-\tau \right) ^{\frac{\gamma \left( 2-\alpha
+\beta \right) }{2}-1}\left\vert \partial _{t}u_{1}\left( \tau \right)
\right\vert _{\alpha }d\tau \left\Vert u\right\Vert _{C\left( \left[ 0,T%
\right] ;V_{\alpha }\right) } \\
& \lesssim \int_{0}^{t}\left( t-\tau \right) ^{\frac{\gamma \left( 2-\alpha
+\beta \right) }{2}-1}\tau ^{\theta -1}d\tau \left\Vert u_{1}\right\Vert
_{Y_{\theta ,\alpha }}\left\Vert u\right\Vert _{C\left( \left[ 0,T\right]
;V_{\alpha }\right) } \\
& \lesssim t^{\frac{\gamma \left( 2-\alpha +\beta \right) }{2}+\theta
-1}\left\Vert z\right\Vert _{C([0,T];L^{2}\left( D\right) )} \\
& \lesssim t^{\theta -1}\left\Vert z\right\Vert _{C([0,T];L^{2}\left(
D\right) )}.
\end{align*}%
Finally, we have%
\begin{equation*}
\left\vert P_{\gamma }\left( t\right) \mathbb{B}z\left( 0\right) \right\vert
_{\alpha }\lesssim t^{\frac{\gamma \left( 2-\alpha +\beta \right) }{2}%
-1}\left\vert \mathbb{B}z\left( 0\right) \right\vert _{\beta =\widetilde{%
\alpha }}\lesssim t^{\frac{\gamma \left( 2-\alpha +\beta \right) }{2}%
-1}\left\Vert z\left( 0\right) \right\Vert _{L^{2}\left( D\right) }\lesssim
t^{\theta -1}\left\Vert z\right\Vert _{Z_{\rho ,\infty }}.
\end{equation*}%
Collecting the previous estimates, from (\ref{u-diff3}) we find that%
\begin{equation}
\left\vert v\left( t\right) \right\vert _{\alpha }\lesssim t^{\theta
-1}\left\Vert z\right\Vert _{Z_{\rho ,\infty }}+\int_{0}^{t}\left( t-\tau
\right) ^{\frac{\gamma \left( 2-\alpha +\beta \right) }{2}-1}\left\vert
v\left( \tau \right) \right\vert _{\alpha }d\tau .  \label{u-diff4}
\end{equation}%
Application of the Gronwall inequality (see Lemma \ref{A2}) then yields%
\begin{equation*}
\left\vert v\left( t\right) \right\vert _{\alpha }\lesssim t^{\theta
-1}\left\Vert z\right\Vert _{Z_{\rho ,\infty }},
\end{equation*}%
from which we can immediately infer (\ref{LC2}). The proof is complete.
\end{proof}

\begin{remark}
\emph{Part (ii) of Theorem \ref{state-map} implies that the solution
operator $\mathcal{S}$ is (Lipschitz) continuous when viewed as a mapping
from $Z_{\rho ,\infty }$ into $Y_{\theta ,\alpha }$ (with $\theta =\gamma
\left( 2-\alpha +\beta \right) /2\in \left( 0,1\right) $). Indeed, (\ref{LC2}%
) yields%
\begin{equation}
\left\Vert u_{1}-u_{2}\right\Vert _{Y_{\theta ,\alpha }}\leq K_{3}\left\Vert
z_{1}-z_{2}\right\Vert _{Z_{\rho ,\infty }},  \label{LC3}
\end{equation}%
for some $K_{3}>0$, depending only on $K_{1},K_{2}.$ }
\end{remark}

Now we define the cost functional, i.e., $J(u,z):=J_{1}(u)+J_{2}(z)$.
Consider $D_{1}$ and $D_{2}$ to be the effective domains of the (proper)
functionals $J_{1}$ and $J_{2}$, respectively. We let $J_{1}:X_{1}%
\rightarrow (-\infty ,+\infty ]$ and $J_{2}:X_{2}:=Z_{\rho ,\infty
}\rightarrow (-\infty ,+\infty ]$ (with the convention that $J_{i}\left(
u\right) =+\infty ,$ for $u\in X_{i}\backslash D_{i},$ $i=1,2$), and as a
result we can write the reduced minimization problem 
\begin{equation}
\min_{z\in Z_{ad}}\mathcal{J}(z):=J_{1}(\mathcal{S}(z))+J_{2}(z)=J\left( 
\mathcal{S}\left( z\right) ,z\right) .  \label{RPS}
\end{equation}%
First, we assume there is an admissible set $Z_{ad}$ which is a convex and
closed subset of $Z_{\rho ,\infty }$. We impose the following (specific)
assumptions on $J_{1}$ and $J_{2}$.

\begin{enumerate}
\item[(H5)] $J_{1}:D_{1}:=Y_{\rho ,\widetilde{\alpha }}\rightarrow \mathbb{R}
$ is weakly lower semicontinuous (for some $\rho >0$); $J_{2}:D_{2}:=Z_{ad}%
\rightarrow \mathbb{R}$ is convex, lower-semicontinuous and the level set $%
\left\{ z\in Z_{ad}:J_{2}\left( z\right) \leq \kappa \right\} $ is bounded
for some $\kappa \in \mathbb{R}$.
\end{enumerate}

\begin{theorem}[\textbf{Existence of optimal controls}]
\label{thm:exist} Let the assumptions of Corollary~\ref{main2} hold. Assume
in addition that (H5) holds. Then the optimal control problem \eqref{RPS}
admits a solution.
\end{theorem}

\begin{proof}
We begin by noticing an {\textsf{infimizing sequence}} $\{z_{n}\}_{n\in {%
\mathbb{N}}}$ always exists (see \cite[pg. 84]%
{HAttouch_GButtazzo_GMichaille_2014a}). Let $\{z_{n}\}_{n\in {\mathbb{N}}}$
be a{~\textsf{infimizing}} sequence, that is, $z_{n}\in Z_{ad}$ and $u_{n}=%
\mathcal{S}(z_{n})$, for $n\in {\mathbb{N}}$, are such that $\mathcal{J}%
(z_{n})\rightarrow j$ as $n\rightarrow \infty $. Moreover, $u_{n}\left(
0\right) =u_{0n}\in V_{\alpha }$ are such that $u_{0n}\rightharpoonup u_{0}$
weakly in $V_{\alpha }$. Notice that 
\begin{equation*}
Z_{ad}\subseteq \mathcal{X}:=W^{1,1+\lambda }\left( \left( 0,T\right)
;L^{2}\left( D\right) \right) ,
\end{equation*}%
where $\mathcal{X}$ is reflexive. Since $Z_{ad}$ is a closed and convex
subspace of $\mathcal{X}$, $Z_{ad}$ is also reflexive, and therefore, by
taking a subsequence if necessary, we may assume that $\{z_{n}\}_{n\in {%
\mathbb{N}}}$ converges weakly in the space $Z_{\rho ,\infty },$ to some $%
z_{\ast }\in Z_{ad}\subset Z_{\rho ,\infty }$ (since $Z_{ad}$ is weakly
compact in the topology of $Z_{\rho ,\infty };$ see \cite[Proposition 3.2.8
and Theorem 3.2.1]{HAttouch_GButtazzo_GMichaille_2014a}).

Next, we aim to show that the state $\{u_{n}\}_{n\in {\mathbb{N}}}$
converges, as $n\rightarrow \infty $, to some $u_{\ast }$ in a suitable
sense, and that $(u_{\ast },z_{\ast })$ satisfies the state equation $%
u_{\ast }=\mathcal{S}(z_{\ast })$. More precisely, $z_{\ast }$ becomes the
desired optimal control for the problem, owing to the (weak)
lower-sequential semicontinuity of the cost functional $\mathcal{J}$. By
virtue of the proof of Corollary \ref{main2}, we observe that $u_{n}$ is
bounded uniformly (with respect to $n\in {\mathbb{N}}$),%
\begin{equation}
u_{n}\in C\left( \left[ 0,T\right] ;V_{\alpha }\right)  \label{mbed2}
\end{equation}%
and%
\begin{equation}
u_{n}\in W^{1,1+\xi }\left( \left( 0,T\right) ;V_{\alpha }\right) \cap
L^{\sigma =1+\xi }\left( (0,T);V_{\widetilde{\alpha }+2}\right) \overset{c}{%
\hookrightarrow }L^{\sigma }\left( (0,T);V_{\alpha }\right) ,  \label{mbed}
\end{equation}%
since $V_{\widetilde{\alpha }+2}\overset{c}{\hookrightarrow }V_{\alpha
}\hookrightarrow V_{\widetilde{\alpha }}=V_{\beta }$,  and $\widetilde{%
\alpha }\leq \alpha <\widetilde{\alpha }+2$ (owing to $\left( I+A\right)
^{-1}$ being compact in $L^{2}\left( \Omega \right) $). We recall that the
embedding in (\ref{mbed}) is also compact due to the (standard)
Aubin-Lions-Simon compactness lemma. It follows that, as $n\rightarrow
\infty ,$%
\begin{equation*}
u_{n}\rightarrow u_{\ast }\text{ strongly in }L^{\sigma }\left(
(0,T);V_{\alpha }\right) .
\end{equation*}%
Together with (H3), this strong convergence implies that%
\begin{equation*}
f\left( u_{n}\right) \rightarrow f\left( u_{\ast }\right) \text{ strongly in 
}L^{\sigma }((0,T);V_{\beta }),
\end{equation*}%
which is enough to pass to the limit in a standard way, in the sequence of
mild solutions%
\begin{equation}
u_{n}=S_{\gamma }\left( t\right) u_{0n}+\int_{0}^{t}P_{\gamma }\left( t-\tau
\right) f\left( u_{n}\left( \tau \right) \right) d\tau
+\int_{0}^{t}P_{\gamma }\left( t-\tau \right) \mathbb{B}z_{n}\left( \tau
\right) d\tau .  \label{seq-un}
\end{equation}%
Indeed, while it is easy to pass to the limit in the first and last summands
on the right-hand side of (\ref{seq-un}), for the second convolution term we
have%
\begin{equation*}
\left\vert P_{\gamma }\ast \left( f\left( u_{n}\right) -f\left( u_{\ast
}\right) \right) \right\vert _{V_{\beta }}\leq t^{\left( \gamma -1\right) 
\frac{\sigma }{\xi }+1}\left\Vert f\left( u_{n}\right) -f\left( u_{\ast
}\right) \right\Vert _{L^{\sigma }\left( 0,T;V_{\beta }\right) },
\end{equation*}%
for all $0<\delta \leq t\leq T$. In particular, we have established the
strong convergence in $C((0,T];V_{\beta }),$ of 
\begin{equation*}
\int_{0}^{t}P_{\gamma }\left( t-\tau \right) f\left( u_{n}\left( \tau
\right) \right) d\tau \rightarrow \int_{0}^{t}P_{\gamma }\left( t-\tau
\right) f\left( u_{\ast }\left( \tau \right) \right) d\tau ,\mbox{ as }
n\to\infty
\end{equation*}%
for any $T<T_{\max }$. Therefore, there holds in $V_{\beta },$ for almost
all $t\in \left( 0,T\right) ,$%
\begin{equation*}
u_{\ast }=S_{\gamma }\left( t\right) u_{0}+\int_{0}^{t}P_{\gamma }\left(
t-\tau \right) f\left( u_{\ast }\left( \tau \right) \right) d\tau
+\int_{0}^{t}P_{\gamma }\left( t-\tau \right) \mathbb{B}z_{\ast }\left( \tau
\right) d\tau .
\end{equation*}%
Due to (\ref{mbed2}), clearly $u_{\ast }\in C\left( \left[ 0,T\right]
;V_{\alpha }\right) $; in fact, one may conclude as in the proof of
Corollary \ref{main2} that $u_{\ast }\in Y_{\theta ,\alpha }\subseteq
Y_{\theta ,\widetilde{\alpha }}$ (for each $\alpha \geq \widetilde{\alpha }$%
) is a solution in the sense of Theorem \ref{main1}$.$

The element $z_{\ast }$ is the right candidate for the optimal control.
Indeed, $z_{\ast }$ is the minimizer. We first notice that since $J_{2}$ is
convex, proper, and lower-semicontinuous, therefore it is weakly
lower-semicontinuous with respect to the $X_{2}$-norm\ topology (see \cite[%
Theorem~3.3.3]{HAttouch_GButtazzo_GMichaille_2014a}). We have 
\begin{align*}
\inf_{z\in Z_{ad}}\mathcal{J}(z)& =\liminf_{n\rightarrow \infty }\mathcal{J}%
(z_{n})\geq \liminf_{n\rightarrow \infty }\mathcal{J}_{1}(\mathcal{S}%
(z_{n}))+\liminf_{n\rightarrow \infty }\mathcal{J}_{2}(z_{n}) \\
& \geq \mathcal{J}_{1}(\mathcal{S}(z_{\ast }))+\mathcal{J}_{2}(z_{\ast
})=j(z_{\ast }).
\end{align*}%
The proof is complete.
\end{proof}

Our next goal is to show differentiability of the control-to-state operator.
We begin with another assumption on $f.$

\begin{enumerate}
\item[(H4bis)] The nonlinearity $f\in C^{2,1}\left( \mathbb{R}\right) $
induces (for a fixed $u\in V_{\alpha }$) a bounded\footnote{%
Note that boundedness of $b_{u}$ is a consequence of (\ref{b-map}).}
nonlinear form%
\begin{equation*}
b_{u}\left( v,w\right) :=\left( \partial _{u}^{2}f\left( u\right) \right)
vw:V_{\alpha }\times V_{\alpha }\rightarrow V_{\beta },
\end{equation*}%
such that, for every $u_{1},u_{2}\in V_{\alpha },$ satisfying $\left\vert
u_{i}\right\vert _{\alpha }\leq R$ ($i=1,2$), there is a constant $C_{R}>0$
such that, 
\begin{equation}
\left\vert b_{u_{1}}\left( v,w\right) -b_{u_{2}}\left( v,w\right)
\right\vert _{\beta }\leq C_{R}\left\vert v\right\vert _{\alpha }\left\vert
w\right\vert _{\alpha }\left\vert u_{1}-u_{2}\right\vert _{\alpha }.
\label{b-map}
\end{equation}
\end{enumerate}

Suppose now that $z_{\ast }\in Z_{ad}$ is a local minimizer for the control
problem, and let $u_{\ast }=\mathcal{S}\left( z_{\ast }\right) $ be the
associated state. We consider, for a fixed $h\in \mathcal{U}$, the
linearized system:%
\begin{equation}
\partial _{t}^{\gamma }\eta \left( t\right) =-A\eta \left( t\right)
+f^{^{\prime }}\left( u_{\ast }\left( t\right) \right) \eta \left( t\right) +%
\mathbb{B}h\text{, }\;\;\eta \left( 0\right) =0,\text{ in }\Omega .
\label{lin-sys2}
\end{equation}%
We now show that problem (\ref{lin-sys2}) admits for every $h\in Z_{\rho
,\infty },$ a unique solution $\eta \in Y_{\theta ,\alpha }$ (in the sense
of Corollary \ref{main2}), and that the linear mapping 
\begin{equation}
\Phi :Z_{\rho ,\infty }\rightarrow Y_{\theta ,\alpha },\text{ }h\mapsto \eta
:=\eta ^{h}  \label{linear-map}
\end{equation}%
is continuous from $Z_{\rho ,\infty }$ into $Y_{\theta ,\alpha }$. Namely,
there is a constant $K_{4}>0$ such that%
\begin{equation}
\left\Vert \eta \right\Vert _{Y_{\theta ,\alpha }}\leq K_{4}\left\Vert
h\right\Vert _{Z_{\rho ,\infty }}.  \label{LC4}
\end{equation}

\begin{lemma}
\label{lem-lc}Let the assumptions of Corollary~\ref{main2} hold, and in
addition, assume (H4) and (H4bis). Then the above statement for %
\eqref{linear-map} holds. Moreover, $\eta \in L^{1+\xi }\left( (0,T);V_{2+%
\widetilde{\alpha }}\right) $ and $\partial _{t}^{\gamma }\eta \in L^{1+\xi
}\left( (0,T);V_{\widetilde{\alpha }}\right) ,$ for the same value $\xi >0$
defined previously.
\end{lemma}

\begin{proof}
The existence of a (unique) mild solution follows exactly along the lines of
Theorem \ref{prelim}, and is based on the formula%
\begin{equation*}
\eta \left( t\right) =\int_{0}^{t}P_{\gamma }\left( t-\tau \right)
f^{^{\prime }}\left( u_{\ast }\left( \tau \right) \right) \eta \left( \tau
\right) d\tau +\int_{0}^{t}P_{\gamma }\left( t-\tau \right) \mathbb{B}%
h\left( \tau \right) d\tau .
\end{equation*}%
We skip the (basic) details for the sake of brevity, and focus mainly on the
stability estimate (\ref{LC4}). Since $\eta $ is also differentiable for
almost all $t\in \left( 0,T\right) ,$ we have%
\begin{align*}
\partial _{t}\eta \left( t\right) & =\int_{0}^{t}P_{\gamma }\left( t-\tau
\right) f^{^{\prime }}\left( u_{\ast }\left( \tau \right) \right) \partial
_{t}\eta \left( \tau \right) d\tau \\
& +\int_{0}^{t}P_{\gamma }\left( t-\tau \right) f^{^{\prime }}\left( u_{\ast
}\left( \tau \right) \right) \partial _{t}\eta \left( \tau \right) d\tau \\
& +\int_{0}^{t}P_{\gamma }\left( t-\tau \right) \mathbb{B\partial }%
_{t}h\left( \tau \right) d\tau +P_{\gamma }\left( t\right) \mathbb{B}h\left(
0\right).
\end{align*}%
This is due to the fact that if $L$ is continuous at $t=0$ and $L$ is of
bounded variation on $\left( 0,T\right) $, we have that 
\begin{equation*}
\frac{d}{dt}\left( P_{\gamma }\ast L\right) \left( t\right) =P_{\gamma
}\left( t\right) L\left( 0\right) +\left( P_{\gamma }\ast \partial
_{t}L\right) \left( t\right) ,\text{ for }t>0.
\end{equation*}%
Following the basic argument developed in the proof of (\ref{LC1}) (see
Theorem \ref{state-map}), it is first easy to see that%
\begin{equation}
\left\Vert \eta \right\Vert _{C\left( \left[ 0,T\right] ;V_{\alpha }\right)
}\leq K_{5}\left\Vert h\right\Vert _{C([0,T];L^{2}\left( D\right) )},
\label{est-10}
\end{equation}%
for some $K_{5}>0$ independent of $h$. Similarly, arguing as in the proof of
(\ref{u-diff3})-(\ref{u-diff4}), owing to the boundedness of the form $%
b_{u}:V_{\alpha }\times V_{\alpha }\rightarrow V_{\beta }$, one has%
\begin{equation}
\left\vert \partial _{t}\eta \left( t\right) \right\vert _{\alpha }\lesssim
t^{\theta -1}\left\Vert h\right\Vert _{Z_{\rho ,\infty }}+\int_{0}^{t}\left(
t-\tau \right) ^{\frac{\gamma \left( 2-\alpha +\beta \right) }{2}%
-1}\left\vert \partial _{t}\eta \left( \tau \right) \right\vert _{\alpha
}d\tau .  \label{est-11}
\end{equation}%
The application of Gronwall lemma to (\ref{est-11}), together with (\ref%
{est-10}), implies the desired conclusion (\ref{LC4}). The final regularity
on $\eta $ can be deduced as in the proof of Corollary \ref{main2} (see also
the proof of Proposition \ref{reg-lin} below).
\end{proof}

\begin{lemma}
\label{lem-lc2}Let the assumptions of Lemma~\ref{lem-lc} hold. Then the
following statements hold:

\begin{enumerate}
\item[(i)] Let $z_{\ast }(=z)\in \mathcal{U}$ be arbitrary. Then the
control-to-state mapping $\mathcal{S}:Z_{\rho ,\infty }\rightarrow Y_{\theta
,\alpha }$ is (Frechet) differentiable at $z_{\ast }$, and the frechet
derivative $d\mathcal{S}\left( z_{\ast }\right) $ is given by $d\mathcal{S}%
\left( z_{\ast }\right) \left( h\right) =\eta ,$ where for any given $h\in
Z_{\rho ,\infty }$, the function $\eta $ denotes the solution of the
linearized system (\ref{lin-sys2}).

\item[(ii)] The mapping $d\mathcal{S}\left( z_{\ast }\right) :\mathcal{%
U\rightarrow L}\left( Z_{\rho ,\infty },Y_{\theta ,\alpha }\right) ,$ $%
z_{\ast }\mapsto d\mathcal{S}\left( z_{\ast }\right) $ is Lipschitz
continuous on $\mathcal{U}$ in the following sense: there exists a constant $%
K_{6}>0$ such that for all $z_{1},z_{2}\in \mathcal{U}$ and all $h\in 
\mathcal{U}$ the following estimate holds: 
\begin{equation*}
\left\Vert d\mathcal{S}\left( z_{1}\right) h-d\mathcal{S}\left( z_{2}\right)
h\right\Vert _{Y_{\theta ,\alpha }}\leq K_{6}\left\Vert
z_{1}-z_{2}\right\Vert _{Z_{\rho ,\infty }}\left\Vert h\right\Vert _{Z_{\rho
,\infty }}.
\end{equation*}
\end{enumerate}
\end{lemma}

\begin{proof}
We begin with (i). Let $z\in \mathcal{U}$ be arbitrarily chosen and let $u=%
\mathcal{S}\left( z\right) $ be the associated solution to the state system.
Since $\mathcal{U}$ is open in $Z_{\rho ,\infty },$ there is $\zeta >0$ such
that for any $h\in Z_{\rho ,\infty }$ with $\left\Vert h\right\Vert
_{Z_{\rho ,\infty }}\leq \zeta $ there holds $z+h\in \mathcal{U}$. In what
follow, we also consider solutions $u^{h}=\mathcal{S}\left( z+h\right) ,$
for such variations $h\in \mathcal{U}$. Next, we let $v^{h}:=u^{h}-u-\eta
^{h},$ where $\eta =\eta ^{h}$ denotes the unique solution to the linearized
system associated with a given $h\in \mathcal{U}$. Our goal is to show that $%
\left\Vert v^{h}\right\Vert _{Y_{\theta ,\alpha }}=o(\left\Vert h\right\Vert
_{Z_{\rho ,\infty }})$, as $\left\Vert h\right\Vert _{Z_{\rho ,\infty
}}\rightarrow 0.$ To this end, we notice that $v^{h}$ satisfies 
\begin{equation*}
v^{h}\left( t\right) =\int_{0}^{t}P_{\gamma }\left( t-\tau \right) \Big(%
f(u^{h}\left( \tau \right) )-f(u\left( \tau \right) )-f^{^{\prime }}(u\left(
\tau \right) )\eta \left( \tau \right)\Big )d\tau ,
\end{equation*}%
for almost all $t\in \left( 0,T\right) \subset \left( 0,T_{\max }\right) $.
By Taylor's theorem \cite[Theorem 30.1.3]{BB}, we have for almost all $t\in
\left( 0,T\right) ,$%
\begin{equation}
f(u^{h}\left( t\right) )-f\left( u\left( t\right) \right) -f^{^{\prime
}}(u\left( t\right) )\eta \left( t\right) =f^{^{\prime }}(u\left( t\right)
)v^{h}\left( t\right) +r^{h}\left( t\right) ,  \label{est12}
\end{equation}%
with remainder%
\begin{equation*}
r^{h}\left( t\right) :=\int_{0}^{1}(f^{^{\prime }}(u\left( t\right)
+x(u^{h}-u)\left( t\right) )-f^{^{\prime }}\left( u\left( t\right) \right)
)(u^{h}-u)\left( t\right) dx.
\end{equation*}%
Moreover, $v^{h}\left( 0\right) =0,$ $r^{h}\left( 0\right) =0$, and $v^{h}$
is differentiable a.e. on $\left( 0,T\right) ,$ with%
\begin{align}
\partial _{t}v^{h}\left( t\right) & =\int_{0}^{t}P_{\gamma }\left( t-\tau
\right) b_{u}(\partial _{t}u\left( \tau \right) ,v^{h}\left( \tau \right)
)d\tau +\int_{0}^{t}P_{\gamma }\left( t-\tau \right) g_{u}(\partial
_{t}v^{h}\left( \tau \right) )d\tau  \label{est14} \\
& +\int_{0}^{1}\int_{0}^{t}P_{\gamma }\left( t-\tau \right) \left[
g_{u+x(u^{h}-u)}(\partial _{t}(u^{h}-u))-g_{u}(\partial _{t}(u^{h}-u))\right]
d\tau dx  \notag \\
& +\int_{0}^{1}\int_{0}^{t}P_{\gamma }\left( t-\tau \right) \left[
b_{u+x(u^{h}-u)}(\partial _{t}u,u^{h}-u)-b_{u}(\partial _{t}u,u^{h}-u)\right]
d\tau dx  \notag \\
& +\int_{0}^{1}\int_{0}^{t}P_{\gamma }\left( t-\tau \right)
xb_{u+x(u^{h}-u)}(u^{h}-u,\partial _{t}(u^{h}-u))d\tau dx  \notag \\
& =:Q_{1}+...+Q_{5}.  \notag
\end{align}%
In view of (\ref{est12}), we can also rewrite%
\begin{equation*}
v^{h}\left( t\right) =\int_{0}^{t}P_{\gamma }\left( t-\tau \right)
f^{^{\prime }}(u\left( \tau \right) )v^{h}\left( \tau \right) )d\tau
+\int_{0}^{t}P_{\gamma }\left( t-\tau \right) r^{h}\left( \tau \right)
)d\tau .
\end{equation*}%
We deduce%
\begin{align*}
\left\vert v^{h}\left( t\right) \right\vert _{\alpha }& \leq
C_{R}\int_{0}^{t}|P_{\gamma }\left( t-\tau \right) f^{^{\prime }}(u\left(
\tau \right) )v^{h}\left( \tau \right) )|_{\alpha }d\tau \\
& +C_{R}\int_{0}^{1}\int_{0}^{t}\left( t-\tau \right) ^{\frac{\gamma \left(
2-\alpha +\beta \right) }{2}-1}x|u^{h}\left( \tau \right) -u\left( \tau
\right) |_{\alpha }d\tau dx \\
& \lesssim t^{\frac{\gamma \left( 2-\alpha +\beta \right) }{2}}||v^{h}||_{C(%
\left[ 0,T\right] ;V_{\alpha })}+t^{\frac{\gamma \left( 2-\alpha +\beta
\right) }{2}}||u^{h}-u||_{C(\left[ 0,T\right] ;V_{\alpha })}^{2},
\end{align*}%
so that for small enough $T\ll 1,$ we obtain%
\begin{equation}
||v^{h}||_{C(\left[ 0,T\right] ;V_{\alpha })}\lesssim ||u^{h}-u||_{C(\left[
0,T\right] ;V_{\alpha })}^{2}\lesssim \left\Vert h\right\Vert _{Z_{\rho
,\infty }}^{2},\text{ by (\ref{LC3}).}  \label{est13}
\end{equation}%
The continuation argument exploited in the proof of Lemma \ref{extension}
yields the same estimate on the whole interval $\left( 0,T\right) ,$ for any 
$T<T_{\max }$. It remains to estimate all $Q_{i}$-terms in (\ref{est14}).
The assumptions (H4)-(H4bis) and (\ref{est13}) are mainly exploited in these
estimates. We thus find that%
\begin{align}
\left\vert Q_{1}\left( t\right) \right\vert _{\alpha }& \lesssim
\int_{0}^{t}\left( t-\tau \right) ^{\frac{\gamma \left( 2-\alpha +\beta
\right) }{2}-1}\tau ^{\theta -1}d\tau \left( 1+\left\Vert u\right\Vert _{C(%
\left[ 0,T\right] ;V_{\alpha })}\right) \left\Vert u\right\Vert _{Y_{\Theta
,\varepsilon }}||v^{h}||_{C(\left[ 0,T\right] ;V_{\alpha })}  \label{est15} \\
& \lesssim t^{\frac{\gamma \left( 2-\alpha +\beta \right) }{2}+\theta
-1}\left\Vert h\right\Vert _{Z_{\rho ,\infty }}^{2},  \notag
\end{align}%
and%
\begin{equation*}
\left\vert Q_{2}\left( t\right) \right\vert _{\alpha }\lesssim
\int_{0}^{t}\left( t-\tau \right) ^{\frac{\gamma \left( 2-\alpha +\beta
\right) }{2}-1}\left\vert \partial _{t}v^{h}\left( \tau \right) \right\vert
_{\alpha }d\tau .
\end{equation*}%
Analogously, we obtain that%
\begin{align*}
\left\vert Q_{4}\left( t\right) \right\vert _{\alpha }& \lesssim
\int_{0}^{t}\left( t-\tau \right) ^{\frac{\gamma \left( 2-\alpha +\beta
\right) }{2}-1}\left\vert \partial _{t}u\left( \tau \right) \right\vert
_{\alpha }|(u^{h}-u)\left( \tau \right) |_{\alpha }^{2}d\tau \\
& \lesssim t^{\frac{\gamma \left( 2-\alpha +\beta \right) }{2}+\theta
-1}||u^{h}-u||_{C(\left[ 0,T\right] ;V_{\alpha })}^{2}\left\Vert
u\right\Vert _{Y_{\Theta ,\epsilon }} \\
& \lesssim t^{\frac{\gamma \left( 2-\alpha +\beta \right) }{2}+\theta
-1}\left\Vert h\right\Vert _{Z_{\rho ,\infty }}^{2},
\end{align*}%
\begin{align*}
\left\vert Q_{5}\left( t\right) \right\vert _{\alpha }& \lesssim
\int_{0}^{t}\left( t-\tau \right) ^{\frac{\gamma \left( 2-\alpha +\beta
\right) }{2}-1}\left\vert \partial _{t}(u^{h}-u)\left( \tau \right)
\right\vert _{\alpha }\left\vert (u^{h}-u)\left( \tau \right) \right\vert
_{\alpha }d\tau \\
& \lesssim t^{\frac{\gamma \left( 2-\alpha +\beta \right) }{2}+\theta
-1}\left\Vert h\right\Vert _{Z_{\rho ,\infty }}^{2},
\end{align*}%
as well as%
\begin{equation}
\left\vert Q_{3}\left( t\right) \right\vert _{\alpha }\lesssim t^{\frac{%
\gamma \left( 2-\alpha +\beta \right) }{2}+\theta -1}\left\Vert h\right\Vert
_{Z_{\rho ,\infty }}^{2}.  \label{est16}
\end{equation}%
Once again collecting the previous estimates, we obtain by means of 
Gronwall's lemma that%
\begin{equation}
\sup_{t\in \left[ 0,T\right] }t^{1-\theta }|\partial _{t}v^{h}\left(
t\right) |_{\alpha }\lesssim \left\Vert h\right\Vert _{Z_{\rho ,\infty
}}^{2}.  \label{est17}
\end{equation}%
Combining (\ref{est13})-(\ref{est17}), we finally arrive at the conclusion
(i). The proof of (ii) follows in a similar fashion; we leave the details to
the interested reader.
\end{proof}

It is now straightforward to derive the standard variational inequality that
optimal controls must satisfy.{~}However, for the method to be practical, it
is critical to identify the adjoint equation. We shall proceed on these two
fronts simultaneously. Exploiting first the integration by parts formula %
\eqref{bpf-eq} of Proposition \ref{ibpf}, we introduce a dual problem, that
can be associated with (\ref{lin-sys2}) whenever $h\equiv 0,$ also owing to $%
\eta \left( 0\right) =0,$%
\begin{equation}
\begin{cases}
\partial _{t,T}^{\gamma }w+Aw & =f^{^{\prime }}\left( u_{\ast }\right)
w+k,\quad \mbox{in }Q:=(0,T)\times \Omega , \\ 
I_{t,T}^{1-\gamma }w(T,\cdot ) & =0\quad \;\quad \mbox{in }\Omega .%
\end{cases}
\label{adjoint}
\end{equation}%
Roughly speaking, one identifies $k$ with $d_{u}J\left( \mathcal{S}\left(
z_{\ast }\right) ,z_{\ast }\right) $, where the function $u_{\ast }=\mathcal{%
S}\left( z_{\ast }\right) $ is the state associated with a minimizer $%
z_{\ast }\in Z_{ad}\subset Z_{\rho ,\infty },$ of (\ref{RPS}). Note that (%
\ref{adjoint}) is a (linear) backward in time partial differential equation
for the (right) Riemann-Liouville fractional derivative $\partial
_{t,T}^{\gamma }$. %
%
We now use the time transformation $t\mapsto T-t:=\overline{t}$ in (\ref%
{adjoint}) to set $p\left( t\right) =w\left( T-t\right) $ (which also yields
that $w\left( t\right) =w\left( T-\overline{t}\right) =p\left( \overline{t}%
\right) $) and $k\left( t\right) :=g\left( \overline{t}\right) $. Employing
the basic identities%
\begin{equation}
(I_{t,T}^{1-\gamma }w)\left( t\right) =(I_{0,\overline{t}}^{1-\gamma
}p)\left( \overline{t}\right) ,\text{ }-\frac{d}{dt}=\frac{d}{d\overline{t}},
\label{ide-dual}
\end{equation}%
we can then transform (\ref{adjoint}) into a (forward) problem for the left
Riemann-Liouville derivative $D_{T-t}^{\gamma }=D_{\overline{t}}^{\gamma }.$
In particular, solvability of problem (\ref{adjoint}) turns out to be
related to solvability of the following (generic) initial-value problem%
\begin{equation}
\left\{ 
\begin{array}{ll}
D_{\overline{t}}^{\gamma }p+Ap=f^{^{\prime }}\left( u_{\ast }\right) p+g=:l,
& \overline{t}\in \left( 0,T\right) , \\ 
(I_{0,\overline{t}}^{1-\gamma }p)\left( 0\right) =p_{0}. & 
\end{array}%
\right.  \label{adjoint3}
\end{equation}%
Here, once again in the context of (\ref{adjoint}) $g$ must be equal to $%
d_{u}J\left( \mathcal{S}\left( z_{\ast }\right) ,z_{\ast }\right) $ and $%
p_{0}=0$. The solvability of the linearized problem (\ref{adjoint3}) under
suitable conditions on $l$ has also been investigated in detail by Bajlekova 
\cite[Section 4, Theorem 4.16]{Ba01}. However, we prefer to give a more
direct proof of the solvability here due to the additional summand $%
f^{^{\prime }}\left( u_{\ast }\right) p$. To this end, we recall that if $%
(I_{0,\overline{t}}^{1-\gamma }p)\left( 0\right) =p_{0},$ one has the
following integral solution representation for the above linearized problem:%
\begin{equation}
p\left( t\right) =P_{\gamma }\left( t\right) p_{0}+\int_{0}^{t}P_{\gamma
}\left( t-\tau \right) \Big (f^{^{\prime }}\left( u_{\ast }\left( \tau
\right) \right) p\left( \tau \right) +g\left( \tau \right) \Big )d\tau ,
\label{adjoint4}
\end{equation}%
where $u_{\ast }\left( t\right) \in \mathbb{B}_{R}$ ($\mathbb{B}_{R}$ is a
ball of radius $R$, in the corresponding strong topology of $Y_{\theta
,\alpha };$ see (\ref{Y-space})), and we have dropped the bar from $%
\overline{t}$, for the sake of notational simplicity. We also recall from (%
\ref{Y-space}) that the Banach space%
\begin{equation*}
Y_{\rho ,\widetilde{\alpha }}=\left\{ k\in C(\left[ 0,T\right] ;V_{%
\widetilde{\alpha }}):\left\vert \partial _{t}k\left( t\right) \right\vert _{%
\widetilde{\alpha }}\lesssim t^{\rho -1},\text{ \ a.e. }0<t\leq T\right\} ,
\end{equation*}%
is subject to the (natural) norm%
\begin{equation*}
\left\Vert k\right\Vert _{Y_{\rho ,\widetilde{\alpha }}}:=\left\Vert
z\right\Vert _{C([0,T];V_{\widetilde{\alpha }})}+\sup_{t\in \left[ 0,T\right]
}t^{1-\rho }\left\vert \partial _{t}k\left( t\right) \right\vert _{%
\widetilde{\alpha }},
\end{equation*}%
for some $\rho >0$. We sketch a proof of the subsequent result in the
Appendix (see Section \ref{sec:ap}).

\begin{proposition}
\label{reg-lin}Let $u:=u_{\ast }\left( t\right) \in \mathbb{B}_{R}$, be an
optimal solution of (\ref{abs-par}) in the sense of Corollary \ref{main2},
for some $R>0$ and $t\in \left( 0,T\right) $ with $T\leq T_{\max }$ ($%
T_{\max }>0$ is the maximal existence time for $u$; see Theorem \ref{prelim}%
).

\begin{enumerate}
\item[(i)] If (H4) holds and $k\in L^{q}((0,T);V_{\widetilde{\alpha }}),$
for some $q\in (\frac{2}{\gamma \left( 2-\alpha +\widetilde{\alpha }\right) }%
,\infty ],$ then the linearized problem (\ref{adjoint}) admits a unique mild
solution on $\left( 0,T\right) $. In particular, one has 
\begin{equation*}
w\in C\left( \left[ 0,T\right] ;V_{\alpha }\right) ;\text{ }\{\partial
_{t,T}^{\gamma }w,Aw\}\in C([0,T];V_{\widetilde{\alpha }-\delta }),\text{
for }\frac{2}{\gamma q}<\delta \leq 2.
\end{equation*}%
The variational equation%
\begin{equation*}
\langle \partial _{t,T}^{\gamma }w\left( t\right) +Aw\left( t\right)
,v\rangle _{V_{\widetilde{\alpha }-\delta },V_{-\widetilde{\alpha }+\delta
}}=\langle f^{^{\prime }}\left( u\left( t\right) \right) w\left( t\right)
+k\left( t\right) ,v\rangle _{V_{\widetilde{\alpha }},V_{-\widetilde{\alpha }%
}},
\end{equation*}%
is satisfied, for any $v\in V_{-\widetilde{\alpha }+\delta }\subset V_{-%
\widetilde{\alpha }}$, for almost all $t\in (0,T).$

\item[(ii)] If (H4), (H4bis) hold and additionally $k\in Y_{\rho ,\widetilde{%
\alpha }},$ then the mild solution of (\ref{adjoint}) also satisfies%
\footnote{%
We recall once again that $\theta $ is independent of $\rho >0.$} $w\in
Y_{\theta ,\alpha },$ and%
\begin{equation*}
\langle \partial _{t,T}^{\gamma }w\left( t\right) +Aw\left( t\right)
,v\rangle _{V_{\widetilde{\alpha }},V_{-\widetilde{\alpha }}}=\langle
f^{^{\prime }}\left( u\left( t\right) \right) w\left( t\right) +k\left(
t\right) ,v\rangle _{V_{\widetilde{\alpha }},V_{-\widetilde{\alpha }}},
\end{equation*}%
for any $v\in V_{-\widetilde{\alpha }}$, for almost all $t\in (0,T).$
Furthermore, $w\in L^{1+\xi }\left( (0,T);V_{2+\widetilde{\alpha }}\right) ,$
$\partial _{t,T}^{\gamma }w\in L^{1+\xi }\left( (0,T);V_{\widetilde{\alpha }%
}\right) ,$ for some $\xi =\xi \left( \theta \right) >0.$
\end{enumerate}
\end{proposition}

It follows from the chain rule that the reduced cost functional $\mathcal{J}%
\left( z\right) =J\left( u,z\right) =J\left( \mathcal{S}\left( z\right)
,z\right) $ is~Fr\'echet differentiable at every~$z\in \mathcal{U}$
(provided that $u\mapsto J_{1}\left( u\right) $ and $z\mapsto J_{2}\left(
z\right) $ are continuously (Fr\'echet) differentiable) with Fr\'echet
derivative%
\begin{equation}
d\mathcal{J}\left( z\right) =d_{u}J\left( \mathcal{S}\left( z\right)
,z\right) \circ d\mathcal{S}\left( z\right) +d_{z}J\left( \mathcal{S}\left(
z\right) ,z\right) .  \label{eq:dJ}
\end{equation}%
Equivalently for any $h\in \mathcal{U}$, we have%
\begin{align}
& \int_{0}^{T}\left( d\mathcal{J}\left( z\left( t\right) \right) ,h\left(
t\right) \right) _{L^{2}\left( D\right) }dt  \label{eq:dJ1} \\
& =\int_{0}^{T}\Big[\langle d_{u}J\left( \mathcal{S}\left( z\left( t\right)
\right) ,z\left( t\right) \right) ,d\mathcal{S}\left( z\left( t\right)
\right) h\left( t\right) \rangle _{V_{\widetilde{\alpha }},V_{-\widetilde{%
\alpha }}}+\left( d_{z}J\left( \mathcal{S}\left( z\left( t\right) \right)
,z\left( t\right) \right) ,h\left( t\right) \right) _{L^{2}\left( D\right) }%
\Big]dt,  \notag
\end{align}%
where we have used the differentiability of $\mathcal{S}$ from Lemma~\ref%
{lem-lc2}(i) and the fact that $d\mathcal{S}\left( z\right) \in \mathcal{L}%
(Z_{\rho ,\infty },Y_{\theta ,\widetilde{\alpha }})$ because $Y_{\theta
,\alpha }\subseteq Y_{\theta ,\widetilde{\alpha }}$, for each $a\geq 
\widetilde{\alpha }\geq -1$. Since $d\mathcal{S}\left( z\right) $ is bounded
and linear, it follows that its adjoint $d\mathcal{S}\left( z\right) ^{\ast }\in 
\mathcal{L}(Z_{\rho ,\infty },Y_{\theta ,\widetilde{\alpha }})^{\ast }$ is
well-defined. Consequently, from \eqref{eq:dJ1} we obtain%
\begin{align}
& \int_{0}^{T}\left( d\mathcal{J}\left( z\left( t\right) \right) ,h\left(
t\right) \right) _{L^{2}\left( D\right) }dt  \label{eq:dJ2} \\
& =\int_{0}^{T}\left( d\mathcal{S}\left( z\left( t\right) \right) ^{\ast
}d_{u}J\left( \mathcal{S}\left( z\left( t\right) \right) ,z\left( t\right)
\right) +d_{z}J\left( \mathcal{S}\left( z\left( t\right) \right) ,z\left(
t\right) \right) ,h\left( t\right) \right) _{L^{2}\left( D\right) }dt. 
\notag
\end{align}%
Thus to evaluate $d\mathcal{J}$, we need to identify $d\mathcal{S}\left(
z\left( \cdot \right) \right) ^{\ast }$, we do this next.

\begin{lemma}
\label{lem:dSstar} Let $(z,u)\in Z_{\rho ,\infty }\times Y_{\theta ,\alpha }$
solve the state equation in the sense of Theorem \ref{main1}. For a.e. $t\in %
\left[ 0,T\right] $, the adjoint operator $d\mathcal{S}(z\left( t\right)
)^{\ast }\psi \left( t\right) :V_{-\widetilde{\alpha }}\rightarrow
L^{2}\left( D\right) $ is given by 
\begin{equation*}
d\mathcal{S}(z\left( t\right) )^{\ast }\psi =\mathbb{B}^{\ast }w\left(
t\right) \in L^{2}\left( D\right) .
\end{equation*}%
Furthermore, $w$ solves the linear equation 
\begin{equation}
\begin{cases}
\partial _{t,T}^{\gamma }w+Aw & =f^{^{\prime }}\left( u_{\ast }\right)
w+\psi ,\text{ }t\in \left( 0,T\right) , \\ 
I_{t,T}^{1-\gamma }w(T,\cdot ) & =0.%
\end{cases}
\label{adjoint0}
\end{equation}
\end{lemma}

\begin{proof}
Recall that for each $\widetilde{\alpha }\in \left[ -1,0\right] $, $\eta =d%
\mathcal{S}(z)h$ solves \eqref{lin-sys2}, while $\eta \in Y_{\theta ,\alpha
}\subset Y_{\theta ,\widetilde{\alpha }}$ and%
\begin{equation*}
\eta \in L^{\sigma }\left( (0,T);V_{2+\widetilde{\alpha }}\right) \subseteq
L^{\sigma }\left(( 0,T);V_{-\widetilde{\alpha }}\right) \text{ and }\partial
_{t}^{\gamma }\eta \in L^{1+\xi }\left( (0,T);V_{\widetilde{\alpha }}\right)
\end{equation*}%
(the values $\sigma $ and $\xi $ are given in Corollary \ref{main2}). For
every $\psi \in Y_{\rho ,\widetilde{\alpha }}$ and $h\in Z_{\rho ,\infty }$,
we have that 
\begin{equation}
\int_{0}^{T}\langle \psi \left( t\right) ,d\mathcal{S}(z\left( t\right)
)h\left( t\right) \rangle _{V_{\widetilde{\alpha }},V_{-\widetilde{\alpha }%
}}dt=\int_{0}^{T}\left( d\mathcal{S}(z)^{\ast }\psi \left( t\right) ,h\left(
t\right) \right) _{L^{2}\left( D\right) }dt.  \label{eq:001}
\end{equation}%
Testing \eqref{adjoint0} with $\eta $ solving \eqref{lin-sys2}, in view of
Proposition \ref{reg-lin}(ii) (recall that $w$ is sufficiently smooth), we
obtain%
\begin{align*}
& \int_{0}^{T}\langle \psi \left( t\right) ,d\mathcal{S}(z\left( t\right)
)h\left( t\right) \rangle _{V_{\widetilde{\alpha }},V_{-\widetilde{\alpha }%
}}dt \\
& =\int_{0}^{T}\langle \psi \left( t\right) ,\eta \left( t\right) \rangle
_{V_{\widetilde{\alpha }},V_{-\widetilde{\alpha }}}dt \\
& =\int_{0}^{T}\Big[\langle \partial _{t,T}^{\gamma }w\left( t\right) ,\eta
\left( t\right) \rangle _{V_{\widetilde{\alpha }},V_{-\widetilde{\alpha }%
}}+\langle Aw\left( t\right) -f^{\prime }(u\left( t\right) )w\left( t\right)
,\eta \left( t\right) \rangle _{V_{\widetilde{\alpha }},V_{-\widetilde{%
\alpha }}}\Big]dt.
\end{align*}%
Applying integration by parts in time (see Proposition \ref{ibpf}) and using
the fact that $A$ can be extended to a (self-adjoint) isomorphism acting
from $V_{\widetilde{\alpha }}$ into $V_{-\widetilde{\alpha }}$, we arrive at 
\begin{equation*}
\int_{0}^{T}\langle \psi \left( t\right) ,d\mathcal{S}(z\left( t\right)
)h\left( t\right) \rangle _{V_{\widetilde{\alpha }},V_{-\widetilde{\alpha }%
}}dt=\int_{0}^{T}\left\langle w\left( t\right) ,\partial _{t}^{\gamma }\eta
(t)+A\eta (t)-f^{\prime }(u(t))\eta (t)\right\rangle _{V_{-\widetilde{\alpha 
}},V_{\widetilde{\alpha }}}dt.
\end{equation*}%
Then using \eqref{lin-sys2}, we immediately obtain 
\begin{equation}
\int_{0}^{T}\langle \psi \left( t\right) ,d\mathcal{S}(z\left( t\right)
)h\left( t\right) \rangle _{V_{\widetilde{\alpha }},V_{-\widetilde{\alpha }%
}}dt=\int_{0}^{T}\left\langle w\left( t\right) ,\mathbb{B}h\left( t\right)
\right\rangle _{V_{-\widetilde{\alpha }},V_{\widetilde{\alpha }}}dt.
\label{eq:002}
\end{equation}%
The asserted result then follows from \eqref{eq:001} and \eqref{eq:002}.
\end{proof}

Finally, we are ready to state the first order necessary optimality
conditions.

\begin{theorem}
\label{fin-thm-necessary-opt}Let $z_{\ast }\in Z_{ad}$ be a local minimum
for \eqref{RPS} with $u_{\ast }=\mathcal{S}(z_{\ast })$ solving the state
equation. If $w$ solves the adjoint equation \eqref{adjoint0} with $\psi $
replaced by $d_{u}J\left( \mathcal{S}\left( z\right) ,z\right) $, then the
following necessary optimality conditions hold: 
\begin{align}
& \int_{0}^{T}\left( d\mathcal{J}\left( z_{\ast }\left( t\right) \right)
,z\left( t\right) -z_{\ast }\left( t\right) \right) _{L^{2}\left( D\right)
}dt  \notag \\
& =\int_{0}^{T}\left( \mathbb{B}^{\ast }w\left( t\right) +d_{z}J\left( 
\mathcal{S}\left( z_{\ast }\left( t\right) \right) ,z_{\ast }\left( t\right)
\right) ,z\left( t\right) -z_{\ast }\left( t\right) \right) _{L^{2}\left(
D\right) }dt \geq 0\label{vari}
\end{align}
for all   $z\in Z_{ad}$.  
\end{theorem}

\begin{proof}
The proof immediately follows from the convexity of $Z_{ad}$ and the assumed
differentiability of $\mathcal{J}$. Towards this end, using Lemma~\ref%
{lem:dSstar}, we can write \eqref{eq:dJ2} equivalently as the right-hand
side of (\ref{vari}), where $w$ solves the adjoint equation \eqref{adjoint0}
with $\psi $ replaced by $d_{u}J\left( \mathcal{S}\left( z\right) ,z\right)$.
\end{proof}

\section{A global regularity result for energy solutions}

\label{ss:global}

Our goal in this section is to describe the proper regularity conditions
necessary to obtain globally defined bounded solutions, i.e., $T_{\max
}=\infty $ (see Corollary \ref{main2}). We begin with a simple energy
estimate which is a consequence of the Hardy-Littlewood theorem. To this
end, let $H$ be a Hilbert space with its associated inner product $\left(
\cdot ,\cdot \right) $ and norm $\left\vert \cdot \right\vert _{H}$,
respectively.

\begin{proposition}
\label{frac-co} If $u\in W^{1,p}\left( (0,T);H\right) $ with $p\geq \frac{2}{%
2-\gamma },$ then the following holds:%
\begin{equation*}
\int_{0}^{T}\left( \partial _{t}^{\gamma }u\left( t\right) ,\partial
_{t}u\left( t\right) \right) dt\geq \sin \frac{\gamma \pi }{2}%
\sum_{n=1}^{\infty }\left\Vert g_{\left( 1-\gamma \right) /2}\ast \partial
_{t}u_{n}\right\Vert _{L^{2}\left( 0,T\right) }^{2}\left\Vert \psi
_{n}\right\Vert _{H}^{2}\geq 0,
\end{equation*}%
for any $\gamma \in \left( 0,1\right) $ and $T>0.$ Here, $u_{n}\left(
t\right) :=\left( u\left( t\right) ,\psi _{n}\right)_H $ where $(\psi
_{n})_{n\in {\mathbb{N}}}$ is an orthogonal basis for $H.$
\end{proposition}

\begin{proof}
We have in $H,$%
\begin{equation*}
\partial _{t}^{\gamma }u\left( t\right) =\sum_{n=1}^{\infty }\left(
g_{1-\gamma }\ast \partial _{t}u_{n}\right) \left( t\right) \psi _{n},\text{ 
}\partial _{t}u\left( t\right) =\sum_{n=1}^{\infty }\partial _{t}u_{n}\left(
t\right) \psi _{n}.
\end{equation*}%
Clearly, $u_{n}\in W^{1,p}\left( 0,T\right) $. It follows that%
\begin{align*}
\int_{0}^{T}\left( \partial _{t}^{\gamma }u\left( t\right) ,\partial
_{t}u\left( t\right) \right) dt& =\sum_{n=1}^{\infty }\left\Vert \psi
_{n}\right\Vert _{H}^{2}\int_{0}^{T}\left( g_{1-\gamma }\ast \partial
_{t}u_{n}\right) \left( t\right) \partial _{t}u_{n}\left( t\right) dt \\
& \geq \sin \frac{\gamma \pi }{2}\sum_{n=1}^{\infty }\left\Vert g_{\left(
1-\gamma \right) /2}\ast \partial _{t}u_{n}\right\Vert _{L^{2}\left(
0,T\right) }^{2}\left\Vert \psi _{n}\right\Vert _{H}^{2}
\end{align*}%
where the last bound follows from \cite[Corollary 2.1]{TYZ}. The proof is
finished.
\end{proof}

We need a crucial regularity assumption on the operator $\mathbb{B}$ in what
follows. For the sake of convenience, we also define $F\left( u\right)
=\int_{0}^{u}f\left( \tau\right) d\tau.$

\begin{enumerate}
\item[(H6)] Let $z\in Z_{\rho ,\infty }$ and $\mathbb{B}\in \mathcal{L}%
\left( L^{2}(D);V_{\widetilde{\alpha }}\right) $ with $u_{0}\in V_{\beta
+2}\subset V_{\alpha =1}$ such that $F\left( u_{0}\right) \in L^{1}\left(
\Omega \right) $. Assume also that (H3) is satisfied with $\alpha =1,$ in $%
\emph{either}$ $\emph{one}$ of the following regimes: (i) $0<\beta <%
\widetilde{\alpha }\leq 1=\alpha $; (ii) $0<\widetilde{\alpha }\leq \beta
\leq 1=\alpha $.
\end{enumerate}

The following global regularity for $\gamma \in \left( 0,1\right) $ is the
main result of the section. Let $\kappa :=\beta ,$ when (i) holds and $%
\kappa :=\widetilde{\alpha },$ if (ii) holds, respectively. Then set $\theta
:=\frac{\gamma }{2}\left( 1+\kappa \right) .$

\begin{theorem}
\label{thm-reg}Let (H6) hold and assume\footnote{%
Strictly speaking, this assumption can be easily relaxed depending on the
application one has in mind for the problem.} $F\left( s\right) \leq -C_{F},$
for all $s\in \mathbb{R}$, for some $C_{F}\in \mathbb{R}$. Then problem (\ref%
{abs-par}) admits a unique (globally-defined) weak solution $u\in Y_{\theta
,1}$. Namely,%
\begin{equation}
u\in W^{1,\frac{1}{1-\theta }-}\left(( 0,T);V_{1}\right) \cap L^{\frac{1}{%
1-\theta }-}\left(( 0,T);V_{\kappa +2}\right) ,\text{ }_{C}\partial
_{t}^{\gamma }u=\partial _{t}^{\gamma }u\in L^{\frac{1}{1-\theta }-}\left(
(0,T);V_{\kappa }\right) ,  \label{reg-en}
\end{equation}%
for any (fixed) but otherwise arbitrary $T>0.$
\end{theorem}

\begin{proof}
Note that by (H6), $1=\alpha \in I_{\widetilde{\alpha }}\cap I_{\beta }\neq
\varnothing $ and (H2)-(H3) are satisfied since $z\in Z_{\rho ,\infty }.$
Therefore, as in the proof of Corollary \ref{main2}, the problem admits a
unique weak solution satisfying (\ref{reg-en}) for every $0<T<T_{\max }$
(where $T_{\max }>0$ is such that either $T_{\max }=\infty ,$ \emph{or} $%
\lim_{t\rightarrow T_{\max }^{-}}\left\vert u\left( t\right) \right\vert
_{1}=\infty $, if $T_{\max }<\infty $). Observe that $u:\left[ 0,T\right]
\rightarrow V_{1}$ is absolutely continuous, and notice also that $p:=\frac{1%
}{1-\theta }>\frac{2}{2-\gamma }$ if and only if $\theta >\gamma /2,$ which
holds since $\beta >0$ on account\footnote{%
This is the only place where one needs to assume $\beta >0$, in order to
derive the regularity $W^{1,p}\left( (0,T);V_{1}\right) ,$ for $p>2/\left(
2-\gamma \right) $ and effectively exploit Proposition \ref{frac-co}.} of
(i) (the case (ii) is similar)$.$ Thus, Proposition \ref{frac-co} applies
with $H=L^{2}\left( \Omega \right) $, and we obtain%
\begin{align}
& 2\left( \partial _{t}^{\gamma }u\left( t\right) ,\partial _{t}u\left(
t\right) \right)_{L^2(\Omega)} +\frac{d}{dt}\left[ \left\vert u\left(
t\right) \right\vert _{1}^{2}-2\left( F\left( u\left( t\right) \right)
,1\right) -2\left( \mathbb{B}z\left( t\right) ,u\left( t\right) \right) %
\right]  \label{reg-en2} \\
& =-2\left( \mathbb{B\partial }_{t}z\left( t\right) ,u\left( t\right)
\right)_{L^2(\Omega)} ,  \notag
\end{align}%
for almost all $t\in \left( 0,T\right) $. The right hand side of (\ref%
{reg-en2}) is bounded in terms of $C\left\vert u\left( t\right) \right\vert
_{1}t^{\rho -1},$ for some $C>0$ independent of $t,T,$ owing to the fact
that $z\in Z_{\rho ,\infty }$ and $\mathbb{B}\in \mathcal{L}\left(
L^{2}(D);V_{-1}\right) .$ In particular, 
\begin{equation*}
2\left\vert \left( \mathbb{B\partial }_{t}z\left( t\right) ,u\left( t\right)
\right)_{L^2(\Omega)} \right\vert \lesssim \left( 1+\left\vert u\left(
t\right) \right\vert _{1}^{2}\right) t^{\rho -1},\text{ for }0<t\leq T
\end{equation*}%
and%
\begin{equation}
2|\left( \mathbb{B}z\left( t\right) ,u\left( t\right)_{L^2(\Omega)} \right)
|\leq C_{\delta }\left\Vert z\right\Vert _{Z_{\rho ,\infty }}^{2}+\delta
\left\vert u\left( t\right) \right\vert _{1}^{2},\text{ for every }\delta >0.
\label{en3}
\end{equation}%
Set now 
\begin{equation*}
E_{\gamma }\left( t\right) :=C_{T}+\left\vert u\left( t\right) \right\vert
_{1}^{2}-2\left( F\left( u\left( t\right) \right) ,1\right)_{L^2(\Omega)}
-2\left( \mathbb{B}z\left( t\right) ,u\left( t\right) \right)_{L^2(\Omega)} ,
\end{equation*}%
where $C_{T}>0$ is sufficiently large (depending clearly on $z\in Z_{\rho
,\infty }$) such that $E_{\gamma }\geq 0$ on $\left( 0,T\right) $ (this is
possible due to (\ref{en3}), for a sufficiently small $\delta \ll 1$)$.$
Notice also that $E_{\gamma }\left( 0\right) \leq (C_{T}+\left\vert
u_{0}\right\vert _{1}^{2}+\left\Vert F\left( u_{0}\right) \right\Vert
_{L^{1}\left( \Omega \right) })$. We immediately deduce from (\ref{reg-en2})
that, 
\begin{equation*}
\partial _{t}E_{\gamma }\left( t\right) +2\left( \partial _{t}^{\gamma
}u\left( t\right) ,\partial _{t}u\left( t\right) \right)_{L^2(\Omega)} \leq
C_{T}^{^{\prime }}\left( 1+\left\vert u\left( t\right) \right\vert
_{1}^{2}\right) t^{\rho -1}\leq C_{T}^{^{\prime \prime }}E_{\gamma }\left(
t\right) t^{\rho -1}
\end{equation*}%
for some $C_{T}^{^{\prime }},C_{T}^{^{\prime \prime }}>0.$ Integrating the
foregoing inequality over $\left( 0,T\right) ,$ we deduce on account of
Gronwall's lemma, that there exists $C_{T}^{^{\prime \prime \prime }}<\infty
,$ such that%
\begin{equation}
C_{T,\delta }(1+\left\vert u\left( T\right) \right\vert _{1}^{2})\leq
E_{\gamma }\left( T\right) \leq C_{T}^{^{\prime \prime \prime }}E_{\gamma
}\left( 0\right) <\infty ,  \label{en4}
\end{equation}%
for any $T>0$. The energy inequality (\ref{en4}) finally yields in view of
Theorem \ref{prelim} that $T_{\max }=\infty .$ The proof is finished.
\end{proof}

The above proof underlines once again the additional smoothness required of
the sources on the right hand side of the equation in order to be able to
(rigorously) justify the energy equality\footnote{%
Note also that a smoother initial datum $u_{0}\in V_{1+\varepsilon },$ for
some $\varepsilon >0$, is neccesary due to the "loss" of regularity of the
solution flow near $t=0.$} for (\ref{abs-par}). This is in contrast to what
happens in the classical case where generally much less is required on the
sources on the right-hand side (and, which is due to the absence of strongly
singular behavior of the solution operator near $t=0$). This result is
optimal since $p=2$ in Proposition \ref{frac-co} corresponds exactly to the
case when $\gamma =1$. In general, $1<p<2$ whenever $0<\gamma <1$, and the
value of $p$ diminishes toward the value $1$, as $\gamma $ goes to zero, no
matter how smooth the right-hand side turns out to be.\footnote{%
A common mistake we had found in scientific journal publications nowadays is
the fact that many authors generally assume in their definition of weak or
smooth solutions that the $p=2$ regularity can be always achieved for arbitrary $0<\gamma <1$ and when sources are sufficiently smooth.}

\section{Control setting and operator examples}

\label{ss:csoe}

In this section we give some examples of operators, functionals, and
nonlinearities that enter in our framework described in the previous
sections.

\subsection{Cost functionals and admissible set}s

\label{ss:csoe2}

We set, for given functions%
\begin{equation}
z_{Q}\in L^{2}\left( (0,T);L^{2}\left( \Omega \right) \right) ,\text{ }%
z_{\Sigma }\in L^{2}\left( (0,T);L^{2}\left( \partial \Omega \right) \right)
\label{reg1}
\end{equation}%
and constants $a_{i}\geq 0$ (not all identically zero), the cost functional%
\begin{equation}
J_{1}(u)=\frac{a_{1}}{2}\int_{0}^{T}\left\Vert u\left( \cdot ,t\right)
-z_{Q}\left( \cdot ,t\right) \right\Vert _{L^{2}\left( \Omega \right)
}^{2}dt+\frac{a_{2}}{2}\int_{0}^{T}\left\Vert u\left( \cdot ,t\right)
-z_{\Sigma }\left( \cdot ,t\right) \right\Vert _{L^{2}\left( \partial \Omega
\right) }^{2}dt.  \label{J1}
\end{equation}%
Additionally, we let 
\begin{equation}
J_{2}(z)=\frac{\zeta }{2}\Vert z\Vert _{L^{2}((0,T)\times D)}^{2},
\label{J2}
\end{equation}%
where $D=\Omega $ if the control $z$ lies in the interior of $\Omega $, and $%
D=\partial \Omega $ in case the control is placed on the boundary $\partial
\Omega $. Moreover, $\zeta >0$ is a control regularization parameter. We
consider the problem of minimizing the total cost functional $%
J:=J_{1}+J_{2}, $ subject to the constraint $z\in Z_{ad},$ where we define
the admissible control set (with prescribed singular behavior near $t=0$) to
be 
\begin{equation}
Z_{ad}:=\{z\in W^{1,2}((0,T);L^{2}(D))\ :\ \left\Vert \partial
_{t}z\right\Vert _{L^{2}\left( D\right) }\leq Mt^{\rho -1},\mbox{ and
}z_{a}\leq z\leq z_{b},\; \mbox{a.e. in }(0,T)\times D\}.
\label{admin-set}
\end{equation}%
Here, $z_{a},z_{b}\in L^{2}((0,T)\times D)$ with $z_{a}\leq z_{b}$ are
given, and we assume that $M>0,$ $\rho >1/2$. Notice that, $Z_{ad}$ is a
closed and convex subset of $Z_{\rho ,\infty }$. Moreover, $Z_{ad}$ depicts
a generic situation with box constraints $z_{a}$ and $z_{b}$. The hypothesis
(H5) is then satisfied by the above $J.$

Our aim is to formulate necessary optimality conditions for our nonlocal in
time (subdiffusive) problem. We introduce $\mathcal{U}$ as a nonempty open
subset of $Z_{\rho ,\infty }$ which contains $Z_{ad}$; without loss of
generality, we may assume that $\mathcal{U}$ is also open in $%
W^{1,2}((0,T);L^{2}(D))$. Recall that the control-to-state mapping $\mathcal{%
S}:\mathcal{U}\rightarrow Y_{\theta ,\alpha },$ $\mathcal{S}\left( z\right)
=u$ is the unique (variational)\ solution of (\ref{abs-par}) (in the sense
of Corollary \ref{main2}). In the context of (\ref{J1})-(\ref{J2}), the
`reduced' cost functional $\mathcal{J}:\mathcal{U}\rightarrow \mathbb{R}$ is
then given by%
\begin{equation*}
\mathcal{J}\left( z\right) =J\left( u,z\right) :=J_{1}\left( \mathcal{S}%
\left( z\right) \right) +J_{2}\left( z\right) .
\end{equation*}%
As $Z_{ad}$ is convex, the desired necessary condition for optimality is 
\begin{equation}
\left\langle d\mathcal{J}\left( z_{\ast }\right) ,z-z_{\ast }\right\rangle
\geq 0,  \label{oc}
\end{equation}%
for every $z\in Z_{ad}$ (for a proper optimal control $z_{\ast }\in Z_{ad}$%
), provided that $d\mathcal{J}\left( z_{\ast }\right) $ is well-defined (at
least in the Gâteaux sense) in the dual space $\Big(W^{1,2}((0,T);L^{2}(D))%
\Big)^{\ast }$. Following Section \ref{ss:optimal}, it turns out that $%
\mathcal{S}$ is (continuously) Fr\'echet differentiable at $z_{\ast }$ so
that the chain rule can be applied. As we had seen in Lemma \ref{lem-lc2},
this leads to the linearized problem (\ref{lin-sys2}) which can then be
stated for a generic element $h\in \mathcal{U}$. This, in turn, leads to the
fact that the Frechet derivative $d\mathcal{S}\left( z\right) \in \mathcal{L}%
\left( Z_{\rho ,\infty },Y_{\theta ,\alpha }\right) $ exists for (a given
generic) $\rho >1/2$, such that $d\mathcal{S}\left( z\right) h=\eta ,$ where 
$\eta $ is the unique (variational)\ solution of the aforementioned
linearized problem. The latter can be described in detail following the
statement (ii) of Proposition \ref{reg-lin}, provided that some regularity
criteria for $k=k\left( t\right) \in Y_{\pi ,\widetilde{\alpha }}$ is given
for another (generic)$\ $parameter $\pi \in (0,1]$ (to be determined below,
see (\ref{reg2})). We thus can immediately apply the chain rule and exploit
the formula (\ref{eq:dJ}), to find that (\ref{oc}) takes on the form%
\begin{equation}
a_{1}\int_{0}^{T}\left( u_{\ast }-z_{Q},\eta \right) _{L^{2}\left( \Omega
\right) }dt+a_{2}\int_{0}^{T}\left( u_{\ast }-z_{\Sigma },\eta \right)
_{L^{2}\left( \partial \Omega \right) }dt+\zeta \int_{0}^{T}\int_{D}z_{\ast
}hdxdt\geq 0,  \label{osc2}
\end{equation}%
for any given $z\in Z_{ad},$ where the function $\eta $ is the solution of
the linearized problem corresponding to $h=z-z_{\ast }.$ As usual, the final
procedure consists in eliminating $\eta $ in (\ref{osc2}) by exploiting the
`backward-in-time' solution 
\begin{equation*}
w\in L^{1+\xi }\left( (0,T);V_{2+\widetilde{\alpha }}\right) ,\text{ }%
\partial _{t,T}^{\gamma }w\in L^{1+\xi }\left( (0,T);V_{\widetilde{\alpha }%
}\right) ,
\end{equation*}%
of the corresponding adjoint problem (see (ii) of Proposition \ref{reg-lin}%
), but now set%
\begin{equation*}
k:=d_{u}J_{1}\left( \mathcal{S}\left( z_{\ast }\right) ,z_{\ast }\right) =%
\binom{a_{1}\left( u_{\ast }-z_{Q}\right) }{a_{2}\left( u_{\ast }-z_{\Sigma
}\right) }\in Y_{\pi ,\widetilde{\alpha }}\text{.}
\end{equation*}%
Namely, $w$ satisfies%
\begin{equation*}
\langle \partial _{t,T}^{\gamma }w\left( t\right) +Aw\left( t\right)
,v\rangle _{V_{\widetilde{\alpha }},V_{-\widetilde{\alpha }}}=\langle
f^{^{\prime }}\left( u\left( t\right) \right) w\left( t\right) +k\left(
t\right) ,v\rangle _{V_{\widetilde{\alpha }},V_{-\widetilde{\alpha }}},
\end{equation*}%
for any $v\in V_{-\widetilde{\alpha }}$, for almost all $t\in (0,T)$. We
note that for $z_{\ast }\in Z_{ad}\subset Z_{\rho ,\infty },$ it follows
from Lemma \ref{lem-lc2} that $u_{\ast }=\mathcal{S}\left( z_{\ast }\right)
\in Y_{\theta ,\alpha }\subseteq Y_{\theta ,\widetilde{\alpha }}$ (since $%
\alpha \geq \widetilde{\alpha }$). Thus, we must choose $\pi :=\theta $ and
consider further regularity assumptions on the data, i.e.,%
\begin{equation}
z_{Q},z_{\Sigma }\in Y_{\theta ,\widetilde{\alpha }}.  \label{reg2}
\end{equation}%
This allows\footnote{%
By the definition of $C_{1},C_{2},$ (\ref{reg2}) implies only additional
temporal regularity of the data.} to conclude that $k\in Y_{\theta ,%
\widetilde{\alpha }};$ finally, we can eliminate $\eta \in L^{1+\xi }\left(
(0,T);V_{2+\widetilde{\alpha }}\right) \subset L^{1+\xi }\left( (0,T);V_{-%
\widetilde{\alpha }}\right) $ (where $\partial _{t}^{\gamma }\eta \in
L^{1+\xi }\left( (0,T);V_{\widetilde{\alpha }}\right) $). Indeed, notice
that all the operations on the left-hand side of (\ref{osc2}) are now
well-defined in view of (\ref{reg1}) and (\ref{reg2}), respectively, and we
can perform calculations exactly as in the proof of Lemma \ref{lem:dSstar}.

We can state a simpler form of the optimality conditions (\ref{osc2}) in the
context of various examples of diffusion operators. We do that next.

\subsection{The fractional Neumann problem for the Laplacian}

For the sake of simplicity, assume that $\Omega \subset \mathbb{R}^{n},$ $%
n\geq 1,$ has a smooth boundary $\partial \Omega $. In this section denote
by $B:=-\Delta _{\Omega ,N}$ the realization of $(-\Delta )$ in $%
L^{2}(\Omega )$ with the zero Neumann boundary condition. Since $\Omega $ is
assumed to be smooth we have that $D(B)=\{u\in H^{2}(\Omega ):\partial _{\nu
}u=0$ on $\partial \Omega \}$. Thus, fractional powers $A:=B^{s}$\ of order $%
s\in \lbrack 0,1]$ can be defined as usual by the semigroup theory and the
following domain characterization holds:%
\begin{equation*}
X_{2s}:=D\left( \left( B+I\right) ^{s}\right) =\left\{ 
\begin{array}{ll}
\{u\in H^{2s}\left( \Omega \right) :\partial _{\nu }u=0\text{ on }\partial
\Omega \}, & s\in (3/4,1] \\ 
H^{2s}\left( \Omega \right) , & s\in (0,3/4),%
\end{array}%
\right.
\end{equation*}%
(in the case $s=3/4,$ $u\in X_{3/2}\subset H^{3/2}\left( \Omega \right) ,$
and equality does not hold). As usual we equip $D\left( A\right) $ with the $%
H^{2s}$-norm. For each $s\in (0,1]$, we consider an internally controlled
system 
\begin{equation}
\begin{cases}
_{C}\partial _{t}^{\gamma }u+B^{s}u=f(u\left( t\right) )+z,\quad & 
\mbox{in
}Q:=(0,T)\times \Omega , \\ 
&  \\ 
u(0,\cdot )=u_{0}\quad & \mbox{in }\Omega .%
\end{cases}
\label{Lin-eq-RL-L}
\end{equation}%
This can be rewritten as the abstract Cauchy problem%
\begin{equation}
\begin{cases}
_{C}\partial _{t}^{\gamma }u\left( t\right) +Au\left( t\right) =f\left(
u\left( t\right) \right) +\mathbb{B}z\left( t\right) ,\quad & t\in (0,T), \\ 
u(0)=u_{0}, & 
\end{cases}
\label{Cau-eq-RL-L}
\end{equation}%
where $\mathbb{B}=\mathbb{I}$ and $D:=\Omega $. Notice that for $\alpha \in
(0,2]$, $V_{\alpha }=D(A^{\alpha /2})=X_{s\alpha }.$ In what follows, we
thus let $\widetilde{\alpha }=0$, $\beta =-1$ and $s\in (0,1].$

\begin{example}
\label{prot1}One prototype for $f$ is a cubic type of nonlinearity $%
f=-F^{^{\prime }}$\ associated with the double-well potential%
\begin{equation}
F\left( u\right) =c_{1}u^{4}-c_{2}u^{2},\text{ for }c_{2}>c_{1}>0\text{.}
\label{double-well}
\end{equation}%
With the choice (\ref{double-well}), one refers to (\ref{Lin-eq-RL-L}) as an
internal optimal control problem for the subdiffusive Allen-Cahn (or phase
field) type equation. The assumptions (H3), (H4), (H4bis)\ are satisfied by
the cubic nonlinearity provided that $s\left( 3\alpha +1\right) >\frac{n}{2}$
for $s\in (0,1]$ and $\alpha \in \left( 0,1\right) .$
\end{example}

\begin{example}
\label{prot2}When $f$ is a logistic reaction term of the form $ru\left(
1-uK^{-1}\right) $ ($r,K>0$), the associated system (\ref{Cau-eq-RL-L}) is
an internal control problem for the subdiffusive Fisher-KPP equation. The
assumptions (H3), (H4), (H4bis)\ are satisfied by the Fisher-KPP type
logistic source provided that $s\in (0,1]$, $\alpha \in \left( 0,1\right) $
satisfy the condition $s\left( 2\alpha +1\right) >\frac{n}{2}.$
\end{example}

\begin{example}
\label{prot3}We can also apply our results to the associated optimal control
problem for a (subdiffusive) Burger's equation, subject to nonlocal
advection or transport. Indeed, let $f\left( u\right) :=-u$div$\left( J\ast
u\right) =-u(\mathcal{G}\ast u),$ where $\mathcal{G}:=$div$\left( J\right) ,$%
\begin{equation*}
\left( J\ast u\right) \left( x\right) =\int_{\Omega }J\left( x-y\right)
u\left( y\right) dy,\text{ }x\in \Omega ,
\end{equation*}%
for some $J\in W_{loc}^{1,div}\left( \mathbb{R}^{n}\right) :=\left\{ J\in
L_{loc}^{1}\left( \mathbb{R}^{n}\right) :\mathcal{G}\in L_{loc}^{1}\left( 
\mathbb{R}^{n}\right) \right\} $. We justify this definition in the one
dimensional case, $\Omega =\left( -L,L\right) $, $L>0$. Given the Dirac
delta function acting $\delta _{x}$ at a point $x\in \Omega ,$ 
\begin{equation*}
\delta _{x}\left[ u\right] =\int_{\Omega }u\left( y\right) \delta \left(
x-y\right) dy=u\left( x\right) ,\text{ }x\in \Omega ,
\end{equation*}%
as a distribution, $\delta _{x}\in C^{k}$ and $\partial _{x}^{k}\delta _{x}%
\left[ u\right] =\left( -1\right) ^{k}\delta _{x}\left[ \partial _{x}u\right]
=\left( -1\right) ^{k}\partial _{x}^{k}u\left( x\right) ,$ for every
positive integer $k$. Then, we observe that, for $u\in C_{c}^{1}\left(
\Omega \right) $ and $v=-\partial _{y}\left( \delta \left( x-y\right)
\right) ,$ we find that%
\begin{equation*}
\int_{\Omega }v\left( y\right) u\left( y\right) dy=\int_{\Omega }\partial
_{y}u\left( y\right) \delta \left( x-y\right) dy=\partial _{x}u\left(
x\right) .
\end{equation*}%
Thus, whenever $u\in C_{c}^{1}\left( \Omega \right) $, $\partial _{x}u\left(
x\right) $ occurs as an approximation of the convolution $\mathcal{G}\ast u,$
where $\mathcal{G}=-\partial _{y}\delta \left( x-y\right) $ (i.e., $f\left(
u\right) \approx -u\left( u_{x}\right) $). Thus, in general we may replace
any $\partial _{x}u\left( x\right) $ with a convolution $\mathcal{G}\ast u$
to reflect the nonlocal behavior of transport at microscopic levels. The
assumptions (H3), (H4)-(H4bis) apply to the nonlocal nonlinearity provided
that $s\left( 3\alpha +1\right) \geq n$ if $n>2s$; no restrictions are
required when $n\leq 2s.$
\end{example}

However, in what follows we will not take a particular choice for $f\left(
u\right) $ since many of the technical assumptions (H3), (H4), (H4bis) can
be verified directly in applications for such nonlinearities.

\begin{corollary}
\label{main2 copy2}Let $u_{0}\in V_{1}=X_{s}$ and $z\in Z_{\rho ,\infty },$
for some $\rho >1/2.$ Assume (H3) for some $\alpha \in \left( 0,1\right) $
and $\beta =-1$. Then (\ref{Lin-eq-RL-L}) admits a unique weak solution on $%
\left( 0,T_{\max }\right) $ such that $u\in Y_{\theta ,\alpha }$ with $%
\theta :=\frac{\gamma }{2}\left( 1-\alpha \right) .$ The variational equality%
\begin{equation*}
\langle _{C}\partial _{t}^{\gamma }u\left( t\right) +Au\left( t\right)
-z,v\rangle _{V_{-1},V_{1}}=\langle f\left( u\left( t\right) \right)
,v\rangle _{V_{-1},V_{1}},
\end{equation*}%
holds for any $v\in V_{1}$, for almost all $t\in (0,T_{\max }).$ In
particular, 
\begin{equation*}
u\in W^{1,1+\xi }\left( \left( 0,T\right) ;V_{\alpha }\right) ,\text{ }Au\in
L^{1+\xi }\left( (0,T);X_{s}^{\ast }\right) ,\text{ }_{C}\partial
_{t}^{\gamma }u\in L^{1+\xi }\left( (0,T);X_{s}^{\ast }\right) ,
\end{equation*}%
for any $T<T_{\max },$ where $\xi \in (0,\frac{\theta }{1-\theta }).$
\end{corollary}

\begin{proof}
By assumption, (H1)-(H2) are automatically satisfied for the operator $%
\mathbb{B}=\mathbb{I}$ and the initial datum $u_{0}$, since $V_{1}\subset
V_{\alpha }$ and $z\in Z_{\rho ,\infty }.$ Thus the conclusions of Corollary %
\ref{main2} hold.
\end{proof}

The conclusions of Section \ref{ss:optimal} hold as well provided that $f$
satisfies the corresponding hypotheses (H4)-(H4bis) in that section with a
given $\alpha \in \left( 0,1\right) $ and $\beta =-1.$

In what follows we consider the cost functional $J$, defined in Section \ref%
{ss:csoe}, by setting $a_{2}=0,$ $a_{1}>0$ and $\zeta \geq 0$. Next, the
datum $z_{Q}$ is assumed to belong to $Y_{\theta ,0},$ for the same value $%
\theta =\frac{\gamma }{2}\left( 1-\alpha \right) $, as in Corollary \ref%
{main2 copy2}. We take the same admissible set $Z_{ad},$ as defined in (\ref%
{admin-set}).

Consequently, it follows on account of Theorem \ref{fin-thm-necessary-opt}
and the previous considerations of Section \ref{ss:csoe2}, the following.

\begin{theorem}
\label{final-optimal1}Let $z_{\ast }\in Z_{ad}$ be an admissible optimal
control and $u_{\ast }=\mathcal{S}\left( z_{\ast }\right) ,$ the associated
state. The necessary optimal condition (\ref{osc2}) reads%
\begin{equation}  \label{eq:VI_0}
\int_{0}^{T}\int_{\Omega }\left( w+\zeta z_{\ast }\right) \left[ z-z_{\ast }%
\right] dxdt\geq 0,\text{ for all }z\in Z_{ad},
\end{equation}%
where $T<T_{\max }$.
\end{theorem}

\begin{remark}
\label{rem:proj} \textrm{{\textsf{Notice that since $Z_{ad}$ is a closed
convex subset of a Hilbert space,  we have in view of \cite[Theorem~3.3.5]%
{HAttouch_GButtazzo_GMichaille_2014a}, instead of solving the variational
inequality \eqref{eq:VI_0}, we can equivalently find $z_{\ast }$ by
computing the projection of $-\frac{1}{\zeta }w$ onto the set $Z_{ad}$ with
respect to the topology on $Z_{ad}$. However, this projection maybe
challenging to evaluate in general, see for instance \cite%
{HAntil_RHNochetto_PVenegas_2018a} for the $H^{1}$-case where each
projection requires solving a variational inequality itself.}} }
\end{remark}

\subsection{The nonhomogeneous Wentzell-Robin problem for the Laplacian}

\label{sec-wentzell}

Assume that $\Omega $ has a Lipschitz continuous boundary. Let $\beta \in
L^{\infty }(\partial \Omega )$ be such that $\beta (x)\geq \beta _{0}>0$ for 
$\sigma $-a.e. $x\in \partial \Omega $ and for some $\beta _{0}\in {\mathbb{R%
}}$, $\delta \in \{0,1\}$ and 
\begin{equation*}
\mathbb{H}^{1,\delta }(\overline{\Omega }):=\Big\{U=(u,u|_{\partial \Omega
}):\;u\in H^{1}(\Omega )\mbox{ and }\;\delta u|_{\partial \Omega }\in
H^{1}(\partial \Omega )\Big\},
\end{equation*}%
be endowed with the norm 
\begin{equation*}
\Vert u\Vert _{\mathbb{H}^{1,\delta }(\overline{\Omega })}=%
\begin{cases}
\left( \Vert u\Vert _{H^{1}(\Omega )}^{2}+\Vert u\Vert _{H^{1}(\partial
\Omega )}^{2}\right) ^{\frac{1}{2}}\;\; & \mbox{ if }\;\delta =1 \\ 
\left( \Vert u\Vert _{H^{1}(\Omega )}^{2}+\Vert u\Vert _{H^{\frac{1}{2}%
}(\partial \Omega )}^{2}\right) ^{\frac{1}{2}}\;\; & \mbox{ if }\;\delta =0.%
\end{cases}%
\end{equation*}%
Then 
\begin{equation}
\mathbb{H}^{1,0}(\overline{\Omega })\hookrightarrow L^{q}(\Omega )\times
L^{q}(\partial \Omega ),  \label{went-emb-0}
\end{equation}%
with 
\begin{equation}
1\leq q\leq \frac{2(n-1)}{n-2}\;\mbox{ if }\;n>2\;\mbox{ and }1\leq q<\infty
\;\mbox{ if }\;n\leq 2,  \label{q-0}
\end{equation}%
and 
\begin{equation}
\mathbb{H}^{1,1}(\overline{\Omega })\hookrightarrow L^{q}(\Omega )\times
L^{q}(\partial \Omega ),  \label{went-emb-1}
\end{equation}%
with 
\begin{equation}
1\leq q\leq \frac{2n}{n-2}\;\mbox{ if }\;n>2\;\mbox{ and }1\leq q<\infty \;%
\mbox{ if }\;n\leq 2.  \label{q-1}
\end{equation}%
Let $\mathcal{E}_{\delta ,W}$ with $D(\mathcal{E}_{\delta ,W}):=\mathbb{H}%
^{1,\delta }(\overline{\Omega })$ be given by 
\begin{equation}
\mathcal{E}_{\delta ,W}(U,V):=\int_{\Omega }\nabla u\cdot \nabla vdx+\delta
\int_{\partial \Omega }\nabla _{\Gamma }u\cdot \nabla _{\Gamma }vd\sigma
+\int_{\partial \Omega }\beta (x)uvd\sigma .  \label{form-went}
\end{equation}%
Let $\Delta _{\delta ,W}$ be the self-adjoint operator in $L^{2}(\Omega
)\times L^{2}(\partial \Omega )$ associated with $\mathcal{E}_{\delta ,W}$.
That is, 
\begin{equation}
\begin{cases}
D(\Delta _{\delta ,W})=\Big\{U:=(u,u|_{\Gamma })\in \mathbb{H}^{1,\delta }(%
\overline{\Omega }):\;\exists \;F:=(f,g)\in L^{2}(\Omega )\times
L^{2}(\partial \Omega ), \\ 
\qquad \qquad \qquad \mathcal{E}_{\delta ,W}(U,V)=(f,v)_{L^{2}(\Omega
)}+(g,v)_{L^{2}(\partial \Omega )}\;\forall \;V:=(v,v|_{\partial \Omega
})\in \mathbb{H}^{1,\delta }(\overline{\Omega })\Big\} \\ 
\Delta _{\delta ,W}U=F.%
\end{cases}%
\end{equation}

Then $\Delta _{\delta ,W}$ is a realization in $L^{2}(\Omega )\times
L^{2}(\partial \Omega)$ of $\Big(-\Delta,-\Delta_\Gamma\Big)$ with the
generalized Wentzell boundary conditions. More precisely, using an
integration by parts arguments, we have that 
\begin{align*}
D(\Delta_{\delta,W})=\Big\{(u,u|_{\Gamma})\in \mathbb{H}^{1,\delta}(%
\overline{\Omega}):\;&\Delta u\in L^{2}(\Omega )\mbox{ and } \\
&\hfill-\delta\Delta_{\Gamma}(u|_{\partial\Omega})+\partial _{\nu
}u+\beta(u|_{\partial\Omega})\in L^2(\partial \Omega) \Big\},
\end{align*}
and 
\begin{align*}
\Delta _{\delta,W}(u,u|_{\Gamma})=\Big(-\Delta
u,-\delta\Delta_{\Gamma}(u|_{\partial\Omega})+\partial _{\nu
}u+\beta(u|_{\partial\Omega})\Big).
\end{align*}%
We notice that for $1\le q\le\infty$, the space $L^q(\Omega)\times
L^q(\partial\Omega)$ endowed with the norm 
\begin{equation*}
\|(f,g)\|_{L^q(\Omega)\times L^q(\partial\Omega)}= 
\begin{cases}
\left(\|f\|_{L^q(\Omega)}^q+\|g\|_{L^q(\partial\Omega)}^q\right)^{1/q}\;\; & %
\mbox{ if }\; 1\le q<\infty, \\ 
\max\{\|f\|_{L^\infty(\Omega)},\|g\|_{L^\infty(\partial\Omega)}\} & 
\mbox{
if }\; q=\infty,%
\end{cases}%
\end{equation*}
can be identified with $L^q(\overline{\Omega},\mu)$ where the measure $\mu$
on $\overline{\Omega}$ is defined for a measurable set $A\subset\overline{%
\Omega}$ by $\mu(A)=|\Omega\cap A|+\sigma(\partial\Omega\cap A)$. In
addition, we have that the embedding $\mathbb{H}^{1,\delta}(\overline\Omega)%
\hookrightarrow L^2(\overline\Omega,\mu)$ is compact.

For $\overline{\delta }\in \left\{ 0,1\right\} ,$ let us consider the
following semilinear problem:

\begin{equation}
\begin{cases}
_{C}\partial _{t}^{\gamma }u-\Delta u=f(u\left( t,x\right) ),\quad & %
\mbox{in }Q:=(0,T)\times \Omega , \\ 
\overline{\delta }_{C}\partial _{t}^{\gamma }u|_{\partial \Omega }-\delta
\Delta _{\Gamma }(u|_{\partial \Omega })+\partial _{\nu }u+\beta
(u|_{\partial \Omega })=z, & \mbox{in }\Gamma :=(0,T)\times \partial \Omega
), \\ 
u(0,\cdot )=(u_{0},v_{0})\quad & \mbox{in }\overline{\Omega }.%
\end{cases}
\label{Went-eq}
\end{equation}

%
%

The system \eqref{Went-eq} can be written as the following abstract Cauchy
problem 
\begin{equation}
\begin{cases}
\mathbb{K}_{\overline{\delta }}\left( _{C}\partial _{t}^{\gamma }U\right)
-\Delta _{\delta ,W}U=(f(u),z)\;\;\; & \mbox{ in
}(0,T)\times (\Omega \times \partial \Omega ) \\ 
U(0,\cdot )=(u_{0},v_{0}) & \mbox{ in }\overline{\Omega },%
\end{cases}
\label{Went-eq-2}
\end{equation}%
where we have identified $_{C}\partial _{t}^{\gamma }U$ with $(_{C}\partial
_{t}^{\gamma }u|_{\Omega },_{C}\partial _{t}^{\gamma }u|_{\partial \Omega })$%
, and set%
\begin{equation*}
\mathbb{K}_{\overline{\delta }}:=\left( 
\begin{array}{cc}
1 & 0 \\ 
0 & \overline{\delta }%
\end{array}%
\right) .
\end{equation*}

Clearly, the system \eqref{Went-eq-2} can be rewritten as the abstract
Cauchy problem (\ref{abs-par}) when $D:=\partial \Omega $, $A=-\Delta
_{\delta ,W},$ and%
\begin{equation*}
\mathbb{B}=\left( 
\begin{array}{cc}
0 & 0 \\ 
0 & 1%
\end{array}%
\right) \in \mathcal{L(}\left\{ 0\right\} \times L^{2}\left( D\right) ,L^{2}(%
\overline{\Omega },\mu ))
\end{equation*}%
is a fixed control operator (and so, in that case $\widetilde{\alpha }=0$).
The statement of Corollary \ref{main2} then applies with $\alpha \in (0,1)$, 
$\theta =\frac{\gamma }{2}\left( 1-\alpha \right) $, $\widetilde{\alpha }=0,$
$\beta =-1$ and $(u_{0},v_{0})\in \mathbb{H}^{1,\delta }(\overline{\Omega }%
)=V_{1}$, provided that $\left( f\left( u\right) ,0\right) $ is locally
Lipschitz in the sense of (H3). All the results of Section \ref{ss:optimal}
are satisfied as well provided that all the corresponding assumptions
(H4)-(H4bis) are applied\footnote{%
For instance, for $f\left( u\right) =ru\left( 1-u/K\right) ,$ these
assumptions are satisfied provided that $\left( 2\alpha +1\right) >\frac{n}{2%
},$ a condition which only restricts the value of $\alpha \in \left(
0,1\right) $ in higher space dimensions $n\geq 3.$} to the vector $\left(
f\left( u\right) ,0\right) $ (see Examples \ref{prot1}, \ref{prot2}, \ref%
{prot3} with $s=1$).

With this setup in mind, we take $a_{1}=0,$ $a_{2}>0$, $\zeta \geq 0$ and
consider the datum $z_{\Sigma }\in Y_{\theta ,0}$ (as in Section \ref%
{ss:csoe2}). Since $\mathbb{B=B}^{\ast }$, we conclude with the following.

\begin{theorem}
\label{final-optimal2}Let $z_{\ast }\in Z_{ad}$ be an admissible optimal
control and $u_{\ast }=\mathcal{S}\left( z_{\ast }\right) ,$ the associated
state. The necessary optimal condition (\ref{osc2}) for the problem (\ref%
{Went-eq}) reads%
\begin{equation*}
\int_{0}^{T}\int_{\partial \Omega }\left( w_{|_{\partial \Omega }}+\zeta
z_{\ast }\right) \left[ z-z_{\ast }\right] d\sigma dt\geq 0,\text{ for all }%
z\in Z_{ad},
\end{equation*}%
where $T<T_{\max }$. 
\end{theorem}

\textsf{Notice that the comment from Remark~\ref{rem:proj}, also applies to
Theorem~\ref{final-optimal2}.}


\section{Appendix}

\label{sec:ap}

For the sake of completeness, we list here the proofs of some statements
that appear in Section \ref{ss:sem}.

\begin{proof}[\textbf{Proof of Lemma \protect\ref{T1-integral-H1}}]
We employ the application of the contraction mapping principle to the mapping%
\begin{equation*}
\Psi \left( u\right) \left( t\right) :=S_{\gamma }\left( t\right)
u_{0}+\int_{0}^{t}P_{\gamma }\left( t-\tau \right) f\left( u\left( \tau
\right) \right) d\tau +\int_{0}^{t}P_{\gamma }\left( t-\tau \right) \mathbb{B%
}z\left( \tau \right) d\tau
\end{equation*}%
in the ball%
\begin{equation*}
B_{T}:=\left\{ u\in C\left( \left[ 0,T\right] ;V_{\alpha }\right)
:\left\Vert u\right\Vert _{C\left( \left[ 0,T\right] ;V_{\alpha }\right)
}\leq R\right\} .
\end{equation*}%
We first show that $\Psi :B_{T_{\ast }}\rightarrow B_{T_{\ast }},$ for some $%
T_{\ast }>0$\ and (any sufficiently large) $R>0.$ Indeed, on account of the
estimates of Proposition \ref{est-SP}, we have%
\begin{align}
\left\vert \Psi \left( u\right) \left( t\right) \right\vert _{\alpha }& \leq
\left\vert S_{\gamma }\left( t\right) u_{0}\right\vert _{\alpha
}+\int_{0}^{t}\left\vert P_{\gamma }\left( t-\tau \right) f\left( u\left(
\tau \right) \right) \right\vert _{\alpha }d\tau  \label{w1} \\
& +\int_{0}^{t}\left\vert P_{\gamma }\left( t-\tau \right) \mathbb{B}z\left(
\tau \right) \right\vert _{\alpha }d\tau  \notag \\
& \lesssim \left\vert u_{0}\right\vert _{\alpha }+\int_{0}^{t}\left( t-\tau
\right) ^{\frac{\gamma }{2}\left( 2-\alpha +\beta \right) -1}\left\vert
f\left( u\left( \tau \right) \right) \right\vert _{\beta }d\tau  \notag \\
& +\int_{0}^{t}\left( t-\tau \right) ^{\frac{\gamma }{2}\left( 2-\alpha +%
\widetilde{\alpha }\right) -1}\left\vert \mathbb{B}z\left( \tau \right)
\right\vert _{\widetilde{\alpha }}d\tau  \notag \\
& \lesssim \left\vert u_{0}\right\vert _{\alpha }+C_{R}t^{\frac{\gamma }{2}%
\left( 2-\alpha +\beta \right) }\left\Vert u\right\Vert _{C\left( \left[ 0,T%
\right] ;V_{\alpha }\right) }+t^{\frac{\gamma }{2}\left( 2-\alpha +%
\widetilde{\alpha }\right) -1+1/p}\left\Vert z\right\Vert _{L^{q}((0,T_{\ast
});L^{2}(D))}  \notag
\end{align}%
owing to the fact that $\mathbb{B}\in \mathcal{L}\left( L^{2}(D),V_{%
\widetilde{\alpha }}\right) $ and $\frac{\gamma }{2}\left( 2-\alpha +%
\widetilde{\alpha }\right) -1+1/p>0.$ The preceeding estimate then yields
the claim that $\Phi :B_{T_{\ast }}\rightarrow B_{T_{\ast }}$ since $%
\left\vert \Psi \left( u\right) \left( t\right) \right\vert _{\alpha }\leq
R, $ for all $0\leq t\leq T_{\ast }$, provided that $T_{\ast },R>0$ are such
that $R\gtrsim 2\left\vert u_{0}\right\vert _{\alpha }$ and 
\begin{equation*}
C_{R}T_{\ast }^{\frac{\gamma }{2}\left( 2-\alpha +\beta \right) }R+T_{\ast
}^{\frac{\gamma }{2}\left( 2-\alpha +\widetilde{\alpha }\right)
-1+1/p}\left\Vert z\right\Vert _{L^{q}((0,T_{\ast });L^{2}(D))}\lesssim
\left\vert u_{0}\right\vert _{\alpha }.
\end{equation*}%
On the other hand, we can choose a much smaller $T_{\ast }>0,$ such that the
mapping $\Psi :B_{T_{\ast }}\rightarrow B_{T_{\ast }}$ is a contraction.
Indeed, for $u,v\in B_{T_{\ast }}$, by the same standard argument, we have%
\begin{align}
\left\vert \Psi \left( u\right) \left( t\right) -\Psi \left( v\right) \left(
t\right) \right\vert _{\alpha }& \leq \int_{0}^{t}\left\vert P_{\gamma
}\left( t-\tau \right) \left( f\left( u\left( \tau \right) \right) -f\left(
v\left( \tau \right) \right) \right) \right\vert _{\alpha }d\tau  \label{w2}
\\
& \lesssim C_{R}t^{\frac{\gamma }{2}\left( 2-\alpha +\beta \right)
}\left\Vert u-v\right\Vert _{C\left( \left[ 0,T_{\ast }\right] ;V_{\alpha
}\right) }  \notag \\
& \lesssim C_{R}T_{\ast }^{\frac{\gamma }{2}\left( 2-\alpha +\beta \right)
}\left\Vert u-v\right\Vert _{C\left( \left[ 0,T_{\ast }\right] ;V_{\alpha
}\right) },  \notag
\end{align}%
provided that $T_{\ast }$ is small enough such that $C_{R}T_{\ast }^{\frac{%
\gamma }{2}\left( 2-\alpha +\beta \right) }\lesssim 1/2.$ Henceforth, the
existence of a unique fixed point $u\in B_{T_{\ast }}$ for the mapping $\Psi 
$ is an immediate consequence of the contraction mapping principle.

Next, we verify the sense in which the initial datum $u\left( 0\right)
=u_{0} $ is satisfied. Recall that $S_{\gamma }$ is strongly continuous as a
mapping from $V_{\alpha }\rightarrow V_{\alpha }.$ Exploiting the argument
of (\ref{w1}), we find that%
\begin{align}
\left\vert u\left( t\right) -u_{0}\right\vert _{\alpha }& \leq \left\vert
u\left( t\right) -S_{\gamma }\left( t\right) u_{0}\right\vert _{\alpha
}+\left\vert S_{\gamma }\left( t\right) u_{0}-u_{0}\right\vert _{\alpha }
\label{w3} \\
& \leq \left\vert S_{\gamma }\left( t\right) u_{0}-u_{0}\right\vert _{\alpha
}+\int_{0}^{t}\left\vert P_{\gamma }\left( t-\tau \right) f\left( u\left(
\tau \right) \right) \right\vert _{\alpha }d\tau  \notag \\
& +\int_{0}^{t}\left\vert P_{\gamma }\left( t-\tau \right) \mathbb{B}z\left(
\tau \right) \right\vert _{\alpha }d\tau  \notag \\
& \lesssim \left\vert S_{\gamma }\left( t\right) u_{0}-u_{0}\right\vert
_{\alpha }+C_{R}t^{\frac{\gamma }{2}\left( 2-\alpha +\beta \right) }R  \notag
\\
& +t^{\frac{\gamma }{2}\left( 2-\alpha +\widetilde{\alpha }\right)
-1+1/p}\left\Vert z\right\Vert _{L^{q}((0,T_{\ast });L^{2}\left( D\right) )}
\notag
\end{align}%
Therefore, the claim in (\ref{ini-cont}) follows by passing to the limit as $%
t\downarrow 0^{+}$ in (\ref{w3}). The proof is
finished.
\end{proof}

\begin{proof}[\textbf{Proof of Lemma \protect\ref{extension}}]
Let $T^{\star }$ be the time from Lemma \ref{T1-integral-H1}. Fix $\tau >0$
and consider the space 
\begin{align*}
\mathbb{K}:=& \Big\{v\in C([0,T^{\star }+\tau ];V_{\alpha })\text{:} \;
v(\cdot ,t)=u(\cdot ,t)\;\qquad \forall \;t\in \lbrack 0,T^{\star }], \\
& \qquad \qquad\left\vert v(\cdot ,t)-u(\cdot ,T^{\star })\right\vert _{\alpha }\leq
R,\;\;\forall \;t\in \lbrack T^{\star },T^{\star }+\tau ]\Big\}.
\end{align*}%
Define the mapping $\Psi ${\textsf{\ }}on $\mathbb{K}$ by 
\begin{equation*}
\Psi (v)(t)=S_{\gamma }\left( t\right) u_{0}+\int_{0}^{t}P_{\gamma }\left(
t-s\right) f\left( v\left( s\right) \right) ds+\int_{0}^{t}P_{\gamma }\left(
t-\tau \right) \mathbb{B}z\left( \tau \right) d\tau .
\end{equation*}%
Note that $\mathbb{K}$ when endowed with the norm of $C([0,T^{\star }+\tau
];V_{\alpha })$ is a closed subspace of $C([0,T^{\star }+\tau ];V_{\alpha })$%
. We show that $\Psi $ has a fixed point in $\mathbb{K}$.

We first show that $\Psi $ maps $\mathbb{K}$ into $\mathbb{K}$. Indeed, let $%
v\in \mathbb{K}$.

If $t\in \lbrack 0,T^{\star }]$, then $v(\cdot ,t)=u(\cdot ,t)$. Hence $\Psi
(v)(t)=\Psi (u)(t)=u(\cdot ,t)$ and there is nothing to prove. If $t\in
\lbrack T^{\star },T^{\star }+\tau ]$, then%
\begin{align*}
& \left\vert \Psi (v)(t)-u(T^{\star })\right\vert _{\alpha } \\
& \leq \left\vert S_{\gamma }(t)u_{0}-S_{\gamma }(T^{\star
})u_{0}\right\vert _{\alpha } \\
& +\left\vert \int_{0}^{t}P_{\gamma }\left( t-s\right) f\left( v\left(
s\right) \right) ds-\int_{0}^{T_{\ast }}P_{\gamma }\left( T_{\ast }-s\right)
f\left( u\left( s\right) \right) ds\right\vert _{\alpha } \\
& +\int_{T_{\ast }}^{t}\left\vert P_{\gamma }\left( t-s\right) \mathbb{B}%
z\left( \tau \right) \right\vert _{\alpha }ds \\
\leq & \left\vert S_{\gamma }(t)u_{0}-S_{\gamma }(T^{\star
})u_{0}\right\vert _{\alpha } \\
& +\int_{0}^{T_{\ast }}\left\vert \left( P_{\gamma }\left( t-s\right)
-P_{\gamma }\left( T_{\ast }-s\right) \right) f\left( u\left( s\right)
\right) \right\vert _{\alpha }ds \\
& +\int_{T_{\ast }}^{t}\left\vert P_{\gamma }\left( t-s\right) f\left(
v\left( s\right) \right) \right\vert _{\alpha }ds+\int_{T_{\ast
}}^{t}\left\vert P_{\gamma }\left( t-s\right) \mathbb{B}z\left( \tau \right)
\right\vert _{\alpha }ds \\
& =:Q_{1}+Q_{2}+Q_{3}+Q_{4}.
\end{align*}%
Since for every $T\geq 0$, the mapping $t\mapsto S_{\gamma }(t)u_{0}$
belongs to $C([0,T],V_{\alpha })$, we can choose $\tau >0$ small such that
for $t\in \lbrack T^{\star },T^{\star }+\tau ]$, we have 
\begin{equation}
Q_{1}=\left\vert S_{\gamma }(t)u_{0}-S_{\gamma }(T^{\star })u_{0}\right\vert
_{\alpha }\leq \frac{R}{4}.  \label{N1}
\end{equation}%
Proceeding as in the proof of Lemma \ref{T1-integral-H1} we can choose $\tau
>0$ small such that for $t\in \lbrack T^{\star },T^{\star }+\tau ]$, we have 
\begin{align}
Q_{3}& =\int_{T_{\ast }}^{t}\left\vert P_{\gamma }\left( t-s\right) f\left(
v\left( s\right) \right) \right\vert _{\alpha }ds\lesssim C_{R}\tau ^{\frac{%
\gamma }{2}\left( 2-\alpha +\beta \right) }R\leq \frac{R}{4},  \label{N2} \\
Q_{4}& \lesssim \tau ^{\frac{\gamma }{2}\left( 2-\alpha +\widetilde{\alpha }%
\right) -1+1/p}\left\Vert z\right\Vert _{L^{q}((T_{\ast },T_{\ast }+\tau
);L^{2}\left( D\right) )}\leq \frac{R}{4}.  \label{N2bis}
\end{align}%
We next write in view of (\ref{op-SP}),%
\begin{align}
Q_{2}& =\int_{0}^{T_{\ast }}\left\vert \left( P_{\gamma }\left( t-s\right)
-P_{\gamma }\left( T_{\ast }-s\right) \right) f\left( u\left( s\right)
\right) \right\vert _{\alpha }ds  \label{N3} \\
& =\int_{0}^{T_{\ast }}\left\vert \left( \gamma \left( t-s\right) ^{\gamma
-1}-\gamma \left( T_{\ast }-s\right) ^{\gamma -1}\right) \int_{0}^{\infty
}\tau \Phi _{\gamma }\left( \tau \right) T\left( \tau \left( t-s\right)
^{\gamma }\right) d\tau w\left( s\right) \right\vert _{\alpha }ds  \notag \\
& +\int_{0}^{T_{\ast }}\left\vert \gamma \left( T_{\ast }-s\right) ^{\gamma
-1}\int_{0}^{\infty }\tau \Phi _{\gamma }\left( \tau \right) \left( T\left(
\tau \left( t-s\right) ^{\gamma }\right) -T\left( \tau \left( T_{\ast
}-s\right) ^{\gamma }\right) \right) d\tau w\left( s\right) \right\vert
_{\alpha }ds  \notag \\
& =:Q_{21}+Q_{22}.  \notag
\end{align}%
Noting that $w\left( s\right) =f\left( u\left( s\right) \right) \in V_{\beta
}$ for $s\in \lbrack 0,T_{\ast }]$, and recalling the semigroup estimate (%
\ref{s-sp}) for $T\left( t\right) :=\exp \left( -At\right) ,$ we have%
\begin{equation*}
\left\vert \left( \gamma \left( t-s\right) ^{\gamma -1}-\gamma \left(
T_{\ast }-s\right) ^{\gamma -1}\right) \int_{0}^{\infty }\tau \Phi _{\gamma
}\left( \tau \right) T\left( \tau \left( t-s\right) ^{\gamma }\right) d\tau
w\left( s\right) \right\vert _{\alpha }\rightarrow 0
\end{equation*}%
in the limit as $t\rightarrow T_{\ast },$ and there exists a constant $%
C_{R}>0$ (where $\left\vert u\left( s\right) \right\vert _{\alpha }\leq R,$ $%
s\in \left[ 0,T_{\ast }\right] $) such that%
\begin{align}
& \left\vert \left( \gamma \left( t-s\right) ^{\gamma -1}-\gamma \left(
T_{\ast }-s\right) ^{\gamma -1}\right) \int_{0}^{\infty }\tau \Phi _{\gamma
}\left( \tau \right) T\left( \tau \left( t-s\right) ^{\gamma }\right) d\tau
w\left( s\right) \right\vert _{\alpha }  \label{N4} \\
& \lesssim C_{R}\left( T_{\ast }-s\right) ^{\gamma -1}\left( t-s\right) ^{-%
\frac{\gamma }{2}\left( \alpha -\beta \right) }\lesssim C_{R}\left( T_{\ast
}-s\right) ^{\frac{\gamma }{2}\left( 2-\alpha +\beta \right) -1}.  \notag
\end{align}%
Thus by the Lebesgue Dominated Convergence Theorem\footnote{%
Note that the right-hand side of (\ref{N4}) is integrable.}, we can choose $%
\tau >0$ small enough such that for $t\in \lbrack T^{\star },T^{\star }+\tau
]$, there holds $Q_{21}\leq \frac{R}{8}$. A similar argument yields that $%
Q_{22}\leq \frac{R}{8}$ on the interval $\left[ T_{\ast },T_{\ast }+\tau %
\right] ,$ for a sufficiently small $\tau >0.$ Henceforth, all the foregoing
estimates imply that $\left\vert \Psi (v)(t)-u(T^{\star })\right\vert
_{\alpha }\leq R,$ for all $t\in \lbrack T^{\star },T^{\star }+\tau ]$. We
have shown that $\Psi $ maps $\mathbb{K}$ into $\mathbb{K}$.

The second step is to show that $\Psi $ is a contraction on $\mathbb{K}$.
Let $v,w\in \mathbb{K}$. Then%
\begin{equation*}
\Psi (v)(t)-\Psi (w)(t)=\int_{0}^{t}P_{\gamma }\left( t-s\right) \left(
f\left( v\left( s\right) \right) -f\left( w\left( s\right) \right) \right)
ds.
\end{equation*}%
Once again if $t\in \lbrack 0,T^{\star }]$, then%
\begin{equation*}
\left\vert \Psi (v)(t)-\Psi (w)(t)\right\vert _{\alpha }=0\leq \frac{1}{2}%
\left\Vert v(t)-w(t)\right\Vert _{C\left( \left[ 0,T_{\ast }\right]
;V_{\alpha }\right) },
\end{equation*}%
owing to the fact that $v=w=u$ on $\left[ 0,T_{\ast }\right] $. On the other
hand, if $t\in \lbrack T^{\star },T^{\star }+\tau ]$, then proceeding as in
the proof of Lemma \ref{T1-integral-H1}, we find%
\begin{align}
\left\vert \Psi (v)(t)-\Psi (w)(t)\right\vert _{\alpha }& \leq \int_{T_{\ast
}}^{t}\left\vert P_{\gamma }\left( t-s\right) \left( f\left( v\left(
s\right) \right) -f\left( w\left( s\right) \right) \right) \right\vert
_{\alpha }ds  \label{w4} \\
& \lesssim C_{R}\tau ^{\frac{\gamma }{2}\left( 2-\alpha +\beta \right)
}\left\Vert v(t)-w(t)\right\Vert _{C\left( \left[ T_{\ast },T_{\ast }+\tau %
\right] ;V_{\alpha }\right) }  \notag \\
& \leq \frac{1}{2}\left\Vert v(t)-w(t)\right\Vert _{C\left( \left[ T_{\ast
},T_{\ast }+\tau \right] ;V_{\alpha }\right) },  \notag
\end{align}%
provided that $\tau >0$ is small enough. We deduce once again that $\Psi $
is a contraction on $\mathbb{K}$ so that it has a unique fixed point $v$ on $%
\mathbb{K}$. The proof is finished.
\end{proof}

\begin{proof}[\textbf{Proof of Theorem \protect\ref{reg-thm}}]
Let $\delta \in \left( 0,T/2\right) $ be an arbitrarily small number and
consider the right-difference%
\begin{equation*}
Z\left( t,h\right) :=h^{-1}\left( u\left( t+h\right) -u\left( t\right)
\right) ,\text{ for }h\in (0,\delta ]\text{ and }\delta <t\leq T.
\end{equation*}%
Notice that $Z\left( t-h,h\right) $ coincides with the left-difference. We
will derive a uniform estimate for $Z\left( t,h\right) $, while we leave the
details of the uniform estimate for $Z\left( t-h,h\right) $ to the
interested reader. For every mild (continuous)\ solution $u\left( t\right)
\in V_{\alpha }$, the continuous function $Z\left( t,h\right) $ satisfies%
\begin{align}
Z\left( t,h\right) & :=h^{-1}\left( S_{\gamma }\left( t+h\right) -S_{\gamma
}\left( t\right) \right) u_{0}  \label{Z1} \\
& +h^{-1}\int_{t}^{t+h}P_{\gamma }\left( s\right) f\left( u\left(
t+h-s\right) \right) ds  \notag \\
& +h^{-1}\int_{0}^{t}P_{\gamma }\left( t-s\right) \left( f\left( u\left(
s+h\right) \right) -f\left( u\left( s\right) \right) \right) ds  \notag \\
& +h^{-1}\int_{t}^{t+h}P_{\gamma }\left( s\right) \mathbb{B}z\left(
t+h-s\right) ds  \notag \\
& +h^{-1}\int_{0}^{t}P_{\gamma }\left( s\right) \mathbb{B}\left( z\left(
t+s+h\right) -z\left( t-s\right) \right) ds  \notag \\
& =:Z_{1}+Z_{2}+Z_{3}+Z_{4}+Z_{5}.  \notag
\end{align}%
We have $\left\vert Z_{1}\left( t,h\right) \right\vert _{\alpha }\rightarrow
0$, as $h\downarrow 0^{+}$ uniformly in $t\in \left[ \delta ,T\right] $
since $S_{\gamma }\left( t\right) $ is analytic for $t\geq \delta >0$ and $%
S_{\gamma }^{^{\prime }}\left( t\right) u_{0}=P_{\gamma }\left( t\right)
Au_{0},$ $u_{0}\in D\left( A\right) $. Hence, we can pick a suffiiciently
small $h_{0}\leq \delta $ such that, for all$\ 0<h\leq h_{0}$ and $\delta
\leq t\leq T,$%
\begin{equation*}
\left\vert Z_{1}\left( t,h\right) \right\vert _{\alpha }\leq 1+\left\vert
P_{\gamma }\left( t\right) Au_{0}\right\vert _{\alpha }\lesssim 1+t^{\gamma
/2\left( 2-\alpha +\beta \right) -1}\left\vert Au_{0}\right\vert _{\beta
}\leq C_{T}t^{\gamma /2\left( 2-\alpha +\beta \right) -1},
\end{equation*}%
for some positive constant $C_{T}$ that depends clearly on $u_{0}\in
V_{\beta +2}$ but is independent of $t,h$ and $\delta $. Next, on account of
Proposition \ref{est-SP}, we estimate%
\begin{eqnarray}
\left\vert Z_{2}\left( t,h\right) \right\vert _{\alpha } &\lesssim
&h^{-1}\int_{t}^{t+h}s^{\frac{\gamma }{2}\left( 2-\alpha +\beta \right)
-1}\left\vert f\left( u\left( t+h-s\right) \right) \right\vert _{\beta }ds
\label{Z3} \\
&\lesssim &C_{R}h^{-1}\left( \left( t+h\right) ^{\frac{\gamma }{2}\left(
2-\alpha +\beta \right) }-t^{\frac{\gamma }{2}\left( 2-\alpha +\beta \right)
}\right)  \notag \\
&\leq &CC_{R}t^{\frac{\gamma }{2}\left( 2-\alpha +\beta \right) -1},  \notag
\end{eqnarray}%
owing to the fact that $\left( t+h\right) ^{r}-t^{r}\leq ht^{r-1}$ for all $%
h,t>0,$ and any $r>0$. Here (and below), the constant $C>0$ is independent
of $h,$ $t$ and $\delta $. Moreover, since $f$ is a (locally) Lipschitz
mapping from $V_{\alpha }\rightarrow V_{\beta }$, we bound%
\begin{align}
\left\vert Z_{3}\left( t,h\right) \right\vert _{\alpha }& \lesssim
h^{-1}\int_{0}^{t}\left( t-s\right) ^{\frac{\gamma }{2}\left( 2-\alpha
+\beta \right) -1}\left\vert f\left( u\left( s+h\right) \right) -f\left(
u\left( s\right) \right) \right\vert _{\beta }ds  \label{Z4} \\
& \leq CC_{R}\int_{0}^{t}\left( t-s\right) ^{\frac{\gamma }{2}\left(
2-\alpha +\beta \right) -1}\left\vert Z\left( s,h\right) \right\vert
_{\alpha }ds,  \notag
\end{align}%
for $0<h\leq \delta ,$ and for all $\delta \leq t\leq T$. Similarly, we
deduce in a similar fashion to (\ref{w1}), that 
\begin{align*}
\left\vert Z_{4}\left( t,h\right) \right\vert _{\alpha }& \lesssim
h^{-1}\int_{t}^{t+h}s^{\frac{\gamma }{2}\left( 2-\alpha +\widetilde{\alpha }%
\right) -1}\left\vert \mathbb{B}z\left( t+h-s\right) \right\vert _{%
\widetilde{\alpha }}ds \\
& \lesssim h^{-1}\left\Vert z\right\Vert _{L^{q}((0,T);L^{2}\left( D\right)
)}[\left( t+h\right) ^{\frac{\gamma }{2}\left( 2-\alpha +\widetilde{\alpha }%
\right) -1+1/p}-t^{\frac{\gamma }{2}\left( 2-\alpha +\widetilde{\alpha }%
\right) -1+1/p}] \\
& \leq C\left\Vert z\right\Vert _{L^{q}((0,T);L^{2}\left( D\right) )}t^{%
\frac{\gamma }{2}\left( 2-\alpha +\widetilde{\alpha }\right) -2+1/p},
\end{align*}%
recalling that $\frac{\gamma }{2}\left( 2-\alpha +\widetilde{\alpha }\right)
-1+1/p>0,$ with $1/p+1/q=1$. Finally, we find that%
\begin{align}
\left\vert Z_{5}\left( t,h\right) \right\vert _{\alpha }& \lesssim
h^{-1}\int_{0}^{t}s^{\frac{\gamma }{2}\left( 2-\alpha +\widetilde{\alpha }%
\right) -1}\left\vert \mathbb{B}\left( z\left( t+s+h\right) -z\left(
t-s\right) \right) \right\vert _{\widetilde{\alpha }}ds  \label{Z4bis} \\
& \lesssim \int_{0}^{t}s^{\frac{\gamma }{2}\left( 2-\alpha +\widetilde{%
\alpha }\right) -1}\left\vert \frac{z\left( t-s+h\right) -z\left( t-s\right) 
}{h}\right\vert _{0}ds  \notag \\
& \lesssim \int_{0}^{t}s^{\frac{\gamma }{2}\left( 2-\alpha +\widetilde{%
\alpha }\right) -1}\left\Vert \partial _{t}z\left( t-s\right) \right\Vert
_{L^{2}}ds  \notag \\
& \lesssim \int_{0}^{t}s^{\frac{\gamma }{2}\left( 2-\alpha +\widetilde{%
\alpha }\right) -1}\left( t-s\right) ^{\rho -1}ds  \notag \\
& \lesssim t^{(\frac{\gamma }{2}\left( 2-\alpha +\widetilde{\alpha }\right)
+\rho )-1},  \notag
\end{align}%
by assumption (\ref{z-reg}). Collecting the uniform estimates for the right
and left differences, one arrives at the following two inequalities:%
\begin{align*}
\left\vert Z\left( t,h\right) \right\vert _{\alpha }& \leq C_{T,R}t^{\theta
-1}+C_{R}\int_{0}^{t}\left( t-s\right) ^{\frac{\gamma }{2}\left( 2-\alpha
+\beta \right) -1}\left\vert Z\left( s,h\right) \right\vert _{\alpha }ds, \\
\left\vert Z\left( t-h,h\right) \right\vert _{\alpha }& \leq
C_{T,R}t^{\theta -1}+C_{R}\int_{0}^{t}\left( t-s\right) ^{\frac{\gamma }{2}%
\left( 2-\alpha +\beta \right) -1}\left\vert Z\left( s-h,h\right)
\right\vert _{\alpha }ds,
\end{align*}%
where, by (H1), 
\begin{equation}
\theta :=\min \left\{ \frac{\gamma }{2}\left( 2-\alpha +\widetilde{\alpha }%
\right) -\frac{1}{q},\frac{\gamma }{2}\left( 2-\alpha +\beta \right) ,\text{ 
}\frac{\gamma }{2}\left( 2-\alpha +\widetilde{\alpha }\right) +\rho \right\}
>0.  \label{thetabis}
\end{equation}%
We can now apply the Gronwall Lemma \ref{A2} to infer the existence of a
constant $C>0$, independent of $h$, $\delta $ and $t$, such that%
\begin{equation}
\left\vert Z\left( t,h\right) \right\vert _{\alpha }\leq Ct^{\theta -1},%
\text{ }\left\vert Z\left( t-h,h\right) \right\vert _{\alpha }\leq
Ct^{\theta -1},  \label{Z5}
\end{equation}%
for $0<h\leq h_{0}\leq \delta $, and for all $t\in \left[ \delta ,T\right] .$
We can now pass to the limit as $h\downarrow 0^{+}$ in the limsup and liminf
sense in (\ref{Z5}), to deduce that both lower Dini derivatives $\partial
_{+}u\left( t\right) ,$ $\partial _{-}u\left( t\right) $ and both upper Dini
derivatives $\partial ^{+}u\left( t\right) ,$ $\partial ^{-}u\left( t\right) 
$ are bounded (as functions with values in $V_{\alpha }$)\ by $Ct^{\theta
-1} $ for all $\delta \leq t\leq T.$ Since $\delta >0$ was arbitrary, all
four Dini derivatives are bounded (and thus finite) in the range for $%
0<t\leq T.$ By application of the celebrated theorem of \ Denjoy-Young-Saks
(see \cite[Chapter IV, Theorem 4.4]{B78}), the continuous integral solution $%
u:\left[ 0,T\right] \rightarrow V_{\alpha }$ is differentiable for almost
all $0<t\leq T,$ and that all four Dini derivatives are equal to $\partial
_{t}u\left( t\right) $ on the set $t\in (0,T]\setminus E$ (where $E$ is a
null set of Lebesegue measure; in fact $E$ is a set of first category, see 
\cite[Chapter IV, Theorem 4.7]{B78}). In particular, this yields the fact
that 
\begin{equation*}
\left\vert \partial _{t}u\left( t\right) \right\vert _{\alpha }\leq
Ct^{\theta -1},
\end{equation*}%
for allmost all $0<t\leq T,$ and the regularity (\ref{u-reg}) follows.
Finally, the last conclusion of the thereom is an immediate consequence of
Proposition \ref{P1}. The proof of the theorem is finished.
\end{proof}

\begin{proof}[\textbf{Proof of Proposition \protect\ref{reg-lin}}]
(i) Owing to (\ref{adjoint4}), we can work with the transformed equation (%
\ref{adjoint3}) for the couple $\left( p,g\right) $. Briefly, since $%
p_{0}=0, $ the contraction mapping principle can be applied to the operator%
\begin{equation*}
\Upsilon \left( p\right) \left( t\right) :=\int_{0}^{t}P_{\gamma }\left(
t-\tau \right) f^{^{\prime }}\left( u\left( \tau \right) \right) p\left(
\tau \right) d\tau +\int_{0}^{t}P_{\gamma }\left( t-\tau \right) g\left(
\tau \right) d\tau
\end{equation*}%
in the ball%
\begin{equation*}
Q_{T_{\ast }}:=\left\{ p\in C\left( \left[ 0,T_{\ast }\right] ;V_{\alpha
}\right) :\left\Vert p\right\Vert _{C\left( \left[ 0,T\right] ;V_{\alpha
}\right) }\leq M\right\} .
\end{equation*}%
It turns out that one can choose a suffiiciently small time $T_{\ast }\leq T$%
, and a sufficiently large $M>0,$ where%
\begin{equation*}
C_{R}T_{\ast }^{\frac{\gamma }{2}\left( 2-\alpha +\beta \right) }\leq \frac{1%
}{2},\text{ }2T_{\ast }^{\frac{\gamma }{2}\left( 2-\alpha +\widetilde{\alpha 
}\right) -1+1/\overline{q}}\left\Vert g\right\Vert _{L^{q}\left( (0,T_{\ast
});V_{\widetilde{\alpha }}\right) }\leq M,
\end{equation*}%
($\overline{q}$ is conjugate to $q$) so that $\Upsilon $ is a strict
contraction on $Q_{T_{\ast }}$. Indeed, owing to (H4), one has for each $%
t\in \left( 0,T_{\ast }\right) ,$%
\begin{equation*}
\left\vert \Upsilon \left( p\right) \left( t\right) \right\vert _{\alpha
}\leq C_{R}t^{\frac{\gamma }{2}\left( 2-\alpha +\beta \right) }\left\Vert
p\right\Vert _{C(\left[ 0,T\right] ;V_{\alpha })}+t^{\frac{\gamma }{2}\left(
2-\alpha +\widetilde{\alpha }\right) -1+1/\overline{p}}\left\Vert
g\right\Vert _{L^{q}\left(( 0,T_{\ast });V_{\widetilde{\alpha }}\right) }
\end{equation*}%
and%
\begin{equation*}
\left\vert \Upsilon \left( p_{1}\right) \left( t\right) -\Upsilon \left(
p_{2}\right) \left( t\right) \right\vert _{\alpha }\leq C_{R}t^{\frac{\gamma 
}{2}\left( 2-\alpha +\beta \right) }\left\Vert p_{1}-p_{2}\right\Vert _{C(%
\left[ 0,T_{\ast }\right] ;V_{\alpha })}.
\end{equation*}%
As in the proof of Theorem \ref{prelim}, the locally defined solution $p$
can be extended to the whole interval $\left( 0,T\right) ,$ for as long as
an optimal solution $u:=u_{\ast }$ exists on $\left( 0,T\right) $. Moreover,
owing to (H4), one has $\xi :=f^{^{\prime }}\left( u\right) p\in
C([0,T];V_{\beta })$, which in turn implies%
\begin{equation*}
\left[ \left\vert A\left( P_{\gamma }\ast g\right) \left( t\right)
\right\vert _{\widetilde{\alpha }-\delta }+\left\vert A\left( P_{\gamma
}\ast \xi \right) \left( t\right) \right\vert _{\widetilde{\alpha }-\delta }%
\right] \lesssim t^{\frac{\gamma \delta }{2}-\frac{1}{q}}\left( \left\Vert
g\right\Vert _{L^{q}\left(( 0,T);V_{\widetilde{\alpha }}\right) }+\left\Vert
\xi \right\Vert _{L^{q}\left( (0,T);V_{\widetilde{\alpha }}\right) }\right)
\end{equation*}%
since $g\in L^{q}\left(( 0,T);V_{\widetilde{\alpha }}\right) $ and $%
\widetilde{\alpha }=\beta $ (this latter condition is chosen for the sake of
simplicity; see Section \ref{ss:sem}). The rest follows analogously to the
proof of Theorem \ref{T1-integral-H1}.

For the proof of (ii), as in the proof of Theorem \ref{reg-thm}, let%
\begin{equation*}
V\left( t,h\right) :=h^{-1}\left( p\left( t+h\right) -p\left( t\right)
\right) ,\text{ for }h>0\text{ and }0<t\leq T.
\end{equation*}%
For every mild continuous\ solution $p\left( t\right) \in V_{\alpha }$, $%
V\left( t,h\right) $ satisfies%
\begin{align}
V\left( t,h\right) & :=h^{-1}\int_{t}^{t+h}P_{\gamma }\left( s\right) \xi
\left( t+h-s\right) ds  \notag \\
& +h^{-1}\int_{0}^{t}P_{\gamma }\left( t-s\right) \left( \xi \left(
s+h\right) -\xi \left( s\right) \right) ds  \notag \\
& +h^{-1}\int_{t}^{t+h}P_{\gamma }\left( s\right) g\left( t+h-s\right) ds 
\notag \\
& +h^{-1}\int_{0}^{t}P_{\gamma }\left( s\right) \left( g\left( t+s+h\right)
-g\left( t-s\right) \right) ds  \notag \\
& =:V_{1}+V_{2}+V_{3}+V_{4}.  \notag
\end{align}%
As before, we estimate by employing (H4)-(H4bis),%
\begin{align*}
\left\vert V_{1}\left( t,h\right) \right\vert _{\alpha }& \lesssim
h^{-1}\int_{t}^{t+h}s^{\frac{\gamma }{2}\left( 2-\alpha +\beta \right)
-1}\left\vert \xi \left( t+h-s\right) \right\vert _{\beta }ds \\
& \leq C_{R,M}h^{-1}\left( \left( t+h\right) ^{\frac{\gamma }{2}\left(
2-\alpha +\beta \right) }-t^{\frac{\gamma }{2}\left( 2-\alpha +\beta \right)
}\right) \\
& \leq C_{R,M}t^{\frac{\gamma }{2}\left( 2-\alpha +\beta \right) -1}.
\end{align*}%
and%
\begin{align*}
\left\vert V_{2}\left( t,h\right) \right\vert _{\alpha }& \lesssim
\int_{0}^{t}\left( t-s\right) ^{\frac{\gamma }{2}\left( 2-\alpha +\beta
\right) -1}\left\vert f^{^{\prime }}\left( u\left( s+h\right) \right)
V\left( s,h\right) \right\vert _{\beta }ds \\
& +h^{-1}\int_{0}^{t}P_{\gamma }\left( t-s\right) (f^{^{\prime }}\left(
u\left( s+h\right) \right) -f^{^{\prime }}\left( u\left( s\right) \right)
)p\left( s\right) ds \\
& \leq C_{R}\int_{0}^{t}\left( t-s\right) ^{\frac{\gamma }{2}\left( 2-\alpha
+\beta \right) }\left\vert Z\left( s,h\right) \right\vert _{\alpha
}\left\vert p\left( s\right) \right\vert _{\alpha }ds \\
& +C_{R}\int_{0}^{t}\left( t-s\right) ^{\frac{\gamma }{2}\left( 2-\alpha
+\beta \right) -1}\left\vert V\left( s,h\right) \right\vert _{\alpha }ds \\
& \lesssim \int_{0}^{t}\left( t-s\right) ^{\frac{\gamma }{2}\left( 2-\alpha
+\beta \right) }s^{\theta -1}ds+\int_{0}^{t}\left( t-s\right) ^{\frac{\gamma 
}{2}\left( 2-\alpha +\beta \right) }\left\vert V\left( s,h\right)
\right\vert _{\alpha }ds \\
& \lesssim t^{\frac{\gamma }{2}\left( 2-\alpha +\beta \right) +\theta
-1}+\int_{0}^{t}\left( t-s\right) ^{\frac{\gamma }{2}\left( 2-\alpha +\beta
\right) }\left\vert V\left( s,h\right) \right\vert _{\alpha }ds,
\end{align*}%
as well as,%
\begin{align*}
\left\vert V_{3}\left( t,h\right) \right\vert _{\alpha }& \lesssim
h^{-1}\int_{t}^{t+h}s^{\frac{\gamma }{2}\left( 2-\alpha +\widetilde{\alpha }%
\right) -1}\left\vert g\left( t+h-s\right) \right\vert _{\widetilde{\alpha }%
}ds \\
& \lesssim h^{-1}\left\Vert g\right\Vert _{L^{q}((0,T);V_{\sim })}[\left(
t+h\right) ^{\frac{\gamma }{2}\left( 2-\alpha +\widetilde{\alpha }\right)
-1+1/\overline{q}}-t^{\frac{\gamma }{2}\left( 2-\alpha +\widetilde{\alpha }%
\right) -1+1/\overline{q}}] \\
& \leq C\left\Vert g\right\Vert _{L^{q}((0,T);V_{\sim })}t^{\frac{\gamma }{2}%
\left( 2-\alpha +\widetilde{\alpha }\right) -2+1/\overline{q}},
\end{align*}%
recalling that $\frac{\gamma }{2}\left( 2-\alpha +\widetilde{\alpha }\right)
-1+1/\overline{q}>0,$ with $1/\overline{q}+1/q=1$. Finally, we find that%
\begin{align}
\left\vert V_{4}\left( t,h\right) \right\vert _{\alpha }& \lesssim
h^{-1}\int_{0}^{t}s^{\frac{\gamma }{2}\left( 2-\alpha +\widetilde{\alpha }%
\right) -1}\left\vert g\left( t+s+h\right) -g\left( t-s\right) \right\vert _{%
\widetilde{\alpha }}ds  \notag \\
& \lesssim \int_{0}^{t}s^{\frac{\gamma }{2}\left( 2-\alpha +\widetilde{%
\alpha }\right) -1}\left\vert \partial _{t}g\left( t-s\right) \right\vert _{%
\widetilde{\alpha }}ds  \notag \\
& \lesssim \int_{0}^{t}s^{\frac{\gamma }{2}\left( 2-\alpha +\widetilde{%
\alpha }\right) -1}\left( t-s\right) ^{\rho -1}ds  \notag \\
& \lesssim t^{(\frac{\gamma }{2}\left( 2-\alpha +\widetilde{\alpha }\right)
+\rho )-1},  \notag
\end{align}%
since $g\in \widetilde{Z}_{\rho ,\infty }$. Collecting these estimates, for
the same value $\theta >0$ as defined earlier, we obtain%
\begin{equation*}
\left\vert V\left( t,h\right) \right\vert _{\alpha }\lesssim t^{\theta
-1}+\int_{0}^{t}\left( t-s\right) ^{\frac{\gamma }{2}\left( 2-\alpha +\beta
\right) -1}\left\vert V\left( s,h\right) \right\vert _{\alpha }ds.
\end{equation*}%
This implies by application of the Gronwall Lemma \ref{A2},%
\begin{equation*}
\left\vert V\left( t,h\right) \right\vert _{\alpha }\leq Ct^{\theta -1}.
\end{equation*}%
A similar estimate applies to the left-difference $V\left( t-h,h\right) .$
The constant on the right hand side is independent of $h,$ so one can pass
to the limit once again as $h\downarrow 0^{+}.$ The proof is finished.
\end{proof}

\bibliographystyle{abbrv}
\bibliography{refs}

\end{document}